\documentclass[twoside,11pt]{article}%
\usepackage{latexsym}
\usepackage{amsmath,amssymb,epsfig,subfigure}
\usepackage{amsthm,enumerate,verbatim}
\usepackage{amsfonts}
\usepackage{color}
\usepackage{amsmath}
\usepackage{graphicx}
\usepackage{epstopdf}
\usepackage{cite}
\usepackage{algpseudocode}
\usepackage{algorithmicx,algorithm}

\providecommand{\U}[1]{\protect\rule{.1in}{.1in}}
\providecommand{\U}[1]{\protect\rule{.1in}{.1in}}
\newcommand{\diag}{{\rm diag}}

\newcommand{\BE}{\begin{equation}}
\newcommand{\EE}{\end{equation}}

\newcommand{\zd}{\,\mathrm{d}}
\numberwithin{equation}{section}
\newtheorem{proposition}{Proposition}[section]
\newtheorem{theorem}[proposition]{Theorem}
\newtheorem{lemma}[proposition]{Lemma}

\newtheorem{remark}[proposition]{Remark}
\newtheorem{example}[proposition]{Example}

\begin{document}
\title{\bf  Second-order and nonuniform time-stepping schemes for time fractional evolution equations with time-space dependent coefficients}

 \author{Pin Lyu\thanks{Email: plyu@swufe.edu.cn. School of Economic Mathematics, Southwestern University of Finance and Economics, Chengdu, China. This author is supported by the Fundamental Research Funds for the Central Universities (JBK2001002), and the National Natural Science Foundation of China (12071373).}
 \and Seakweng Vong\thanks{Corresponding author.
 Email: swvong@um.edu.mo. Department of Mathematics, University of Macau, Macao, China. This author is
 funded by The Science and Technology Development Fund, Macau SAR (File no. 0005/2019/A)  and
University of Macau (File no. MYRG2018-00047-FST)}
}
\date{}
 \maketitle\normalsize

\begin{abstract}
The numerical analysis of time fractional evolution equations with the second-order elliptic operator including general time-space dependent variable coefficients is challenging, especially when the classical weak initial singularities are taken into account. In this paper, we introduce a concise technique to construct efficient time-stepping schemes with variable time step sizes for two-dimensional time fractional sub-diffusion and diffusion-wave equations with general time-space dependent variable coefficients. By means of the novel technique, the nonuniform Alikhanov type schemes are constructed and analyzed for the sub-diffusion and diffusion-wave problems. For the diffusion-wave problem, our scheme is constructed by employing the recently established symmetric fractional-order reduction (SFOR) method. The unconditional stability of proposed schemes is rigorously  discussed under mild assumptions on variable coefficients and, based on reasonable regularity assumptions and weak time mesh restrictions, the second-order convergence is obtained with respect to discrete $H^1$-norm. Numerical experiments are given to demonstrate the theoretical statements.
\end{abstract}

\section{Introduction}

Fractional differential equations (FDEs) are extremely powerful mathematical tools for the modeling of diverse processes and phenomena which contain memory and hereditary properties, interested readers may refer to  \cite{IP1999,Mainardi2001,FaKS2005,Metzler2000} and references therein for some practical applications of FDEs in physics, biology and chemistry, etc. Finding efficient and accurate numerical solutions of FDEs becomes an increasingly  hot research topic as it is hard to obtain the reliable analytic solution in general.

In this work, we consider numerical analysis of the two-dimensional time fractional evolution equations with general time-space dependent variable coefficients:
\begin{align}\label{eq1}
&{\cal D}_t^\alpha u={\cal A}u +f({\bf x},t), \quad {\bf x}\in\Omega,~t\in(0,T];\\\label{eq2}
%&u({\bf x},t)=0, \quad {\bf x}\in\partial\Omega;\\
&u({\bf x},0)=\varphi({\bf x}), \quad {\bf x}\in\Omega,\quad \mbox{if} \quad \alpha\in(0,1); \\\label{eq3}
&u({\bf x},0)=\phi({\bf x}),~ u_t({\bf x},0)=\psi({\bf x}), \quad {\bf x}\in\Omega, \quad \mbox{if} \quad \alpha\in(1,2);
\end{align}
subject to the homogeneous boundary condition $u({\bf x},t)=0$ for $({\bf x},t)\in\partial\Omega\times(0,T]$, where $\Omega=(x_l,x_r)\times(y_l,y_r)$, ${\bf x}=(x,y)$ and ${\cal A}$ is a linear second-order elliptic operator which is dependent on time and space:
\begin{align}\label{ellipticA}
{\cal A}u:=\left\{a_1({\bf x},t)\partial_{xx}^2+a_2({\bf x},t)\partial_{yy}^2+b_1({\bf x},t)\partial_{x}+b_2({\bf x},t)\partial_{y}+b_3({\bf x},t)\right \} u.
\end{align}
The fractional derivative ${\cal D}_t^\alpha$ in \eqref{eq1} is defined by the Caputo sense:
$${\cal D}_t^{\alpha}u(t):=\int_0^t \omega_{n-\alpha}(t-s)u^{(n)}(s)\zd s  \quad \mbox{with}\quad \omega_{n-\alpha}(t)=\frac{t^{n-1-\alpha}}{\Gamma(n-\alpha)} ,  \quad n=\lceil \alpha \rceil,\quad t>0. $$

With error analyses basing on sufficient smoothness in time of the analytical solutions, various numerical methods  are designed for fractional sub-diffusion or diffusion-wave equations with variable coefficients which are time-space dependent (e.g. \cite{Yang2018}) or only space dependent (e.g.\cite{Cui-JCP2013,VLWjsc2016,ZhangAA2020,ZhaoXu2014,WangYM2016,WeiZhao2020}).
 However, it is well-known that the solution of time fractional initial value problems typically exhibits weak initial singularities.
  Thus most of the traditional time-stepping methods fail to preserve the desired convergence rates in this general and practical situation. In \cite{KoptevaMC2019}, Kopteva discussed the L1-type discretizations on graded time meshes for fractional parabolic equation with classical weak singular solutions, where the second-order elliptic operator  ${\cal L}u=\sum_{k=1}^d\{-\partial_{x_k}a_k(x)\partial_{x_k}u+b_k(x)\partial_{x_k}u\}+c(x)u$  $(d=1,2,3)$ is only space dependent. A second-order convergent method was studied lately in Wei et al.\cite{WeiAML2021},  where the Alikhanov formula \cite{AAA} on  the graded time meshes is considered to deal with the weak initial singularity of the two-dimensional time fractional diffusion equations with the elliptic operator ${\cal L}u=\mbox{div} (a({\bf x})\nabla u)$ which is symmetric and space dependent only.
Recently, the sub-diffusion problems with time-space dependent coefficients and nonsmooth data were  studied in several research works.
In \cite{Mustapha-MC2018}, Mustapha studied a semidiscrete Galerkin finite element method for the time fractional diffusion equations with time-space dependent diffusivity coefficient:
\begin{align}\label{ts-dependent-1}
{\cal D}_t^\alpha u({\bf x},t)=\mbox{div} (a({\bf x},t)\nabla u({\bf x},t))+f({\bf x},t) \quad \mbox{in}~\Omega\times (0,T], ~\alpha\in(0,1),
\end{align}
where $\Omega\subset {\mathbb R}^d~(d\geq 1)$, and  the optimal error bounds in $L^2$- and $H^1$-norms are obtained for both smooth and nonsmooth initial data. Jin, Li and Zhou \cite{Jin-MC2019} then proposed an efficient numerical scheme with the Galerkin finite element method in space and backward Euler convolution quadrature in time for the problem \eqref{ts-dependent-1}. The optimal convergence with first-order temporal accuracy is obtained provided a certain regularity of the solutions is proved for both nonsmooth initial data and incompatible source term. The second-order temporal convergence was further achieved for the convolution quadrature generated by second-order backward differentiation formula with proper correction at the first time step \cite{Jin-NM2020}, where an improved regularity was shown.
We remark that, based on some mild and natural assumptions on $a({\bf x},t)$, the time-space dependent elliptic operator in \eqref{ts-dependent-1} is symmetric and is a particular case (for $d=2$) of ${\cal A}$ in \eqref{ellipticA} (or $L$ in \cite{Yang2018}) because $\mbox{div} (a({\bf x},t)\nabla u({\bf x},t))=a({\bf x},t)\Delta u({\bf x},t)+ \nabla a({\bf x},t) \cdot\nabla u({\bf x},t)$.
There have been many works on the theoretical and numerical study of classical parabolic and hyperbolic equations with general time-space dependent elliptic operator, e.g. \cite{Alinhac,LeeSIAM2013,Luskin1982}.
To the best of our knowledge, taking the weak initial singularity into account, there is no study on the efficient numerical methods for time fractional evolution equations (sub-diffusion and diffusion-wave) where the elliptic operators include general time-space dependent coefficients, i.e., the elliptic operators take the form \eqref{ellipticA}.

In the past few years, numerical methods on nonuniform time meshes are found to be very efficient and thus are of great interests in resolving the weak initial singularities of the time fractional initial value problems \cite{JiLiaoSIAMjsc2020,KoptevaMC2019,LiaoL1,LiaoSecondOrder,Liao-AC2,LiaoYanZhang2018,LyuVong2020diffu-wave,Stynes-SIAM2017,ChenStynesJSC2019,WeiAML2021}. As the operator ${\cal A}$ in \eqref{ellipticA} is  in general non-symmetric and is substantially different from the one in \cite{LiWangXie2020}, this brings challenges in the analysis of the standard nonuniform approximations of \eqref{eq1}--\eqref{eq3}. To tackle the problem, we will introduce a novel and concise technique to study highly accurate numerical methods  for the time fractional evolution equations with general time-space dependent coefficients on nonuniform time meshes. By the proposed technique,  an important estimate, i.e. the inequality \eqref{zQDz2},
can be guaranteed in the analysis of corresponding nonuniform algorithms.
 Our numerical schemes will utilize the Alikhanov formula on possible nonuniform time meshes to approximate the Caputo derivatives.
We recall that, for a given positive integer $N$, the Alikhanov formulas  for the Caputo derivative ${\cal D}_t^{\beta} g(t_{n-\theta})~(0<\beta<1)$
on arbitrary time meshes $0=t_0<t_1<\cdots<t_N=T$ is expressed by the following summation of convolution structure \cite{LiaoSecondOrder}:
 \begin{align}\label{dis-Caputo}
 ({\cal D}_\tau^{\beta} g)^{n-\theta}:=\sum_{k=1}^nA_{n-k}^{(n)}\nabla_{\tau} g^k, \quad \mbox{where}~g^k=g(t_k)~\mbox{and}~\nabla_{\tau} g^k=g^k-g^{k-1},
 \end{align}
 where %$\beta=\alpha$ if $\alpha\in(0,1)$ (for sub-diffusion) and $\beta=\alpha/2$ if $\alpha\in(1,2)$ (for diffusion-wave),  and
 $\theta:=\beta/2$. For simplicity of presentation, the precise formulation of the coefficients $A_{n-k}^{(n)}$ and its corresponding  properties are given in Appendix (Subsection \ref{L1Alikhanov}).  To study the diffusion-wave problem, we further employ a symmetric fractional-order reduction (SFOR) method which was investigated in our very recent work \cite{LyuVong2020diffu-wave}.
\begin{remark}
In the rest of this paper, we always take the setting:
\begin{equation}\nonumber
\beta=\left\{\begin{array}{ll}\smallskip
\alpha,\quad & \mbox{if}~\alpha\in(0,1),~\mbox{i.e.},~\mbox{while concerning the sub-diffusion problem};\\
\alpha/2,~ & \mbox{if}~\alpha\in(1,2),~\mbox{i.e.},~\mbox{while concerning the diffusion-wave problem}.
\end{array}\right.
\end{equation}
\end{remark}
In the construction and analysis of our proposed numerical methods, the variable coefficients involved in ${\cal A}$ are assumed to satisfy two generic conditions: For ${\bf x}\in\Omega,~t\in[0,T]$,
\begin{itemize}
\item[{\bf V1}.] $a_k({\bf x},t)>0$, and $a_k(\cdot,t)\in {\cal C}^1([0,T])$ with $ \left|(a_1)_t/a_1\right|+\left|(a_2)_t/a_2\right|\leq C_p$, for $k=1,2$;
\item[{\bf V2}.] $a_k({\bf x},\cdot)\in {\cal C}^3(\Omega)$ for $k=1,2$, and $\left|b_l({\bf x},t)\right|\leq {C}_l$ for $l=1,2,3$,
 \end{itemize}
 where $C_p$ and $C_l$ are positive constants. We will obtain the second-order $H^1$-norm convergence  (in time and space) of the proposed nonuniform schemes for both the sub-diffusion and diffusion-wave problems under the following assumptions on regularity ($C_u$ is a positive constant):
 For $t\in(0,T]$ and $k=1,2,3$,
  \begin{align}\label{regularity1}
  &  \| u\|_{H^4(\Omega)}\leq C_u,  \quad\mbox{for} \quad  \alpha\in (0,1)\cup(1,2);\\\label{regularity2}
 &\|\partial_t^{(k)}u\|_{H^3(\Omega)}\leq C_u(1+t^{\sigma_1-k}),  \quad \mbox{if}~\alpha\in (0,1); \\\label{regularity3}
 &\|\partial_t^{(k)}{ u}\|_{H^3(\Omega)}\leq C_u(1+t^{\sigma_2-k}),\quad
  \|\partial_t^{(k)}{v}\|_{H^3(\Omega)}\leq C_u(1+t^{\sigma_3-k}) ,\quad \mbox{if}~\alpha\in (1,2),
 \end{align}
 where $v:={\cal D}_t^\beta {\tilde u}$ with  ${\tilde u}:=u-t\psi$, $\sigma_1\in(0,1)\cup(1,2)$, $\sigma_2\in(1,2)\cup(2,3)$ and $\sigma_3\in(1/2,1)\cup(1,2)$;
Furthermore, we impose the weak mesh assumption:
 \begin{itemize}
 \item[{\bf MA.}] There is a constant $C_\gamma>0$ such that $\tau_k\leq C_\gamma\tau\min\{1,t_k^{1-1/\gamma}\}$ for $1\leq k\leq N$, with $t_k\leq C_\gamma t_{k-1}$ and $\tau_k/t_k\leq C_\gamma\tau_{k-1}/t_{k-1}$ for $2\leq k\leq N$,
 \end{itemize}
  where $\gamma\geq1$ is the mesh parameter,   $\tau_k:=t_k-t_{k-1}$ denotes the $k$-th time step size and $\tau:=\max_{1\leq k\leq N}\{\tau_k\}$.

The rest of the paper is organized as follows. In Section \ref{new-technique}, we will introduce a  concise technique which is used to construct and analyze nonuniform schemes for the governing problems. In Section \ref{sub-diffusion}, the numerical scheme which based on the nonuniform Alikhanov formula is proposed for sub-diffusion equation with general variable coefficients, and its stability and second-order convergence are rigorously  discussed with respect to discrete $H^1$-norm. In Section \ref{diffusion-wave}, by applying the SFOR method, the nonuniform Alikhanov type scheme is constructed for the diffusion-wave equation with general variable coefficients. We also show that the scheme is  stable and second-order convergent  in the discrete $H^1$-norm. Numerical examples are given in Section \ref{Numerical} to demonstrate the theoretical statements.
As an appendix, in  Section \ref{Appendix}, the precise definitions of the coefficients of Alikhanov formula,  the proof of inequality \eqref{zQDz2} and the analysis of truncation errors are given.

\section{A  technique for numerical analysis}\label{new-technique}

In this section, we will present a concise technique   to study numerical schemes with variable time step sizes for time fractional evolution equations with general time-space dependent variable coefficients.

Firstly, we show an important lemma which extends the one in \cite[Lemma 4.1]{LiaoGronwall}.
\begin{lemma}\label{zQDz}
For a continuous  (w.r.t. $x$ and $t$)  function $q(x,t)>0$, $x\in(x_l,x_r)\subset{\mathbb R}$, $t\in[0,T]$, we define a diagonal matrix
$${\bf Q}^{(k)}:=\diag\left(q(x_1,t_k),q(x_2,t_k), \cdots,q(x_{m},t_k)\right),\quad m\geq 1, ~k\geq 0,$$
where $t_k\in [0,T]$ with $t_j<t_{j+1}$, and $x_i\in (x_l,x_r)$. Let ${\bf z}^k:=(z_1^k,z_2^k,\cdots,z_m^k)^T$ be a real vector, and ${\bf z}^{n-\theta}:=(1-\theta){\bf z}^{n}+\theta {\bf z}^{n-1}$.
Then
\begin{align}\label{zQDz1}
({\bf z}^{n-\theta})^T{\bf Q}^{(n)}({\cal D}_\tau^\beta {\bf z})^{n-\theta}\geq \frac12\sum_{k=1}^nA_{n-k}^{(n)}\nabla_\tau [({\bf z}^k)^T{\bf Q}^{(n)}{\bf z}^k].
%,~ \mbox{where}~{\bf z}^{n-\theta}=(1-\theta){\bf z}^{n}+\theta {\bf z}^{n-1}.
\end{align}
Moreover, if $q(x,t)$ is {\bf non-increasing w.r.t. $t$ for every fixed $x$}, it holds that
\begin{align}\label{zQDz2}
({\bf z}^{n-\theta})^T{\bf Q}^{(n)}({\cal D}_\tau^\beta {\bf z})^{n-\theta}\geq \frac12\sum_{k=1}^nA_{n-k}^{(n)}\nabla_\tau [({\bf z}^k)^T{\bf Q}^{(k)}{\bf z}^k].
\end{align}
\end{lemma}
\begin{proof}
The inequality \eqref{zQDz1} can be verified according to \cite[Lemma 4.1]{LiaoGronwall}, we move its derivation to the Appendix (Subsection \ref{zQDz0}).

If $q(x,t)$ is non-increasing w.r.t. $t$ for every fixed $x$, we have $({\bf z}^k)^T{\bf Q}^{(n)}{\bf z}^k\leq ({\bf z}^k)^T{\bf Q}^{(k)}{\bf z}^k$ while $k\leq n$. Then
\begin{align*}
&({\bf z}^{n-\theta})^T{\bf Q}^{(n)}({\cal D}_\tau^\beta {\bf z})^{n-\theta}\\
\geq& \frac12\sum_{k=1}^nA_{n-k}^{(n)}\nabla_\tau [({\bf z}^k)^T{\bf Q}^{(n)}{\bf z}^k]\\
=&\frac12\left[A_{0}^{(n)}({\bf z}^n)^T{\bf Q}^{(n)}{\bf z}^n-\sum_{k=1}^{n-1}(A_{n-k-1}^{(n)}-A_{n-k}^{(n)}) ({\bf z}^k)^T{\bf Q}^{(n)}{\bf z}^k- A_{n-1}^{(n)}({\bf z}^0)^T{\bf Q}^{(n)}{\bf z}^0\right]\\
\geq&\frac12\left[A_{0}^{(n)}({\bf z}^n)^T{\bf Q}^{(n)}{\bf z}^n-\sum_{k=1}^{n-1}(A_{n-k-1}^{(n)}-A_{n-k}^{(n)}) ({\bf z}^k)^T{\bf Q}^{(k)}{\bf z}^k- A_{n-1}^{(n)}({\bf z}^0)^T{\bf Q}^{(0)}{\bf z}^0\right]\\
=&\frac12\sum_{k=1}^nA_{n-k}^{(n)}\nabla_\tau [({\bf z}^k)^T{\bf Q}^{(k)}{\bf z}^k].
\end{align*}
\end{proof}
Next, we will utilize Lemma \ref{zQDz} to obtain some  properties for our numerical analysis.
For a continuous (w.r.t ${\bf x}$ and $t$) function $p({\bf x},t)>0$, suppose $p_1=pa_1$ and $p_2=pa_2$, we have
\begin{align*}
{\cal A}u=p^{-1}(p{\cal A}u)&=p^{-1}\left\{p_1\partial_{xx}^2+p_2\partial_{yy}^2+pb_1\partial_x+pb_2\partial_y+pb_3\right\}u\\
&=p^{-1}\left\{\partial_x (p_1\partial_x u)+\partial_y (p_2\partial_y u)+[pb_1-(p_1)_x]\partial_x u +[pb_2-(p_2)_y]\partial_y u+pb_3 u\right\}\\
&=p^{-1}\left[\partial_x (p_1\partial_x u)+\partial_y (p_2\partial_y u)\right]+[b_1-p^{-1}(p_1)_x]\partial_x u +[b_2-p^{-1}(p_2)_y]\partial_y u+b_3 u.
\end{align*}
Some spatial notations are required.  For two positive integers $M_x$ and $M_y$, denote $h_x:=(x_r-x_l)/M_x$ and $h_y:=(y_r-y_l)/M_y$. Define the mesh space $\Omega_h:=\{{\bf x}_h=(x_l+ih_x,y_l+jh_y)|1\leq i\leq M_x-1,1\leq j\leq M_y-1\}$ and $\bar\Omega_h:=\Omega_h\cup \partial \Omega$. For any grid functions $u_h:=\{u_{i,j}=u(x_i,y_j)|(x_i,y_j)\in\bar\Omega_h \}$, the central difference operators are given by
         $$\delta_x u_{i+\frac12,j}:=(u_{i+1,j}-u_{i,j})/{h_x}, ~0\leq i\leq M_x-1;\quad \delta_{\hat x} u_{i,j}:=(u_{i+1,j}-u_{i-1,j})/(2h_x),~1\leq i\leq M_x-1;$$
         and $\delta_y u_{i,j+\frac12}$, $\delta_{\hat y} u_{i,j}$ are defined similarly.

%Define a discrete operator corresponding to ${\cal A}$:
%$${\cal A}_h^{n-\theta}:=(p_h^{n-\theta})^{-1}\left\{\delta_x [(p_1)_h^{n-\theta} \delta_x] +\delta_y [(p_2)_h^{n-\theta} \delta_y]  \right\}+ (p_3)_h^{n-\theta} \delta_{\hat x}+ (p_4)_h^{n-\theta} \delta_{\hat y}+(b_3)_h^{n-\theta},$$
%where $p_3:=b_1-p^{-1}(p_1)_x$ and $p_4:=b_2-p^{-1}(p_2)_y$, with

Denote $p_3:=b_1-p^{-1}(p_1)_x$, $p_4:=b_2-p^{-1}(p_2)_y$, the discrete function
$p_h^{n-\theta}:=p({\bf x}_h,t_{n-\theta})~(0\leq n\leq N)$ with $p_h^{-\theta}:=p({\bf x}_h,t_0)$, and we use similar notations  for $(p_k)_h^{n-\theta} ~(k=1,2,3,4)$ and $(b_3)_h^{n-\theta}$. Then we define a discrete operator corresponding to ${\cal A}$:
$${\cal A}_h^{n-\theta}:=(p_h^{n-\theta})^{-1}\left\{\delta_x [(p_1)_h^{n-\theta} \delta_x] +\delta_y [(p_2)_h^{n-\theta} \delta_y]  \right\}+ (p_3)_h^{n-\theta} \delta_{\hat x}+ (p_4)_h^{n-\theta} \delta_{\hat y}+(b_3)_h^{n-\theta}.$$
Since the numerical schemes and corresponding analysis in the next two sections will be done in matrix form, we define the following matrices (the symbol `$\otimes$' denotes the Kronecker product)
\begin{align*}
&{\bf P}^{n-\theta}:=\diag(p_{1,1}^{n-\theta},\cdots,p_{M_x-1,1}^{n-\theta},p_{1,2}^{n-\theta},\cdots,p_{M_x-1,2}^{n-\theta},\cdots\cdots,p_{1,M_y-1}^{n-\theta},\cdots,p_{M_x-1,M_y-1}^{n-\theta}),\\
&{\bf P}_1^{n-\theta}:=\diag((p_1)_{1/2,1}^{n-\theta},\cdots,(p_1)_{M_x-1/2,1}^{n-\theta},(p_1)_{1/2,2}^{n-\theta},\cdots,(p_1)_{M_x-1/2,2}^{n-\theta},\\
&\qquad\qquad\qquad\cdots\cdots,(p_1)_{1/2,M_y-1}^{n-\theta},\cdots,(p_1)_{M_x-1/2,M_y-1}^{n-\theta}),\\
&{\bf P}_2^{n-\theta}:=\diag((p_2)_{1,1/2}^{n-\theta},\cdots,(p_2)_{M_x-1,1/2}^{n-\theta},(p_2)_{1,3/2}^{n-\theta},\cdots,(p_2)_{M_x-1,3/2}^{n-\theta},\\
&\qquad\qquad\qquad\cdots\cdots,(p_2)_{1,M_y-1/2}^{n-\theta},\cdots,(p_2)_{M_x-1,M_y-1/2}^{n-\theta}),\\
&{\bf A}^{n-\theta}:=(I_y\otimes S_x)^T{\bf P}_1^{n-\theta}(I_y\otimes S_x)+(S_y\otimes I_x)^T{\bf P}_2^{n-\theta}(S_y\otimes I_x),\\
&{\bf B}^{n-\theta}:={\bf P}_3^{n-\theta}[I_y\otimes({\hat S}_x-{\hat S}_x^T)]+{\bf P}_4^{n-\theta}[({\hat S}_y-{\hat S}_y^T)\otimes I_x],\\
&{\bf C}^{n-\theta}:=\diag((b_3)_{1,1}^{n-\theta},\cdots,(b_3)_{M_x-1,1}^{n-\theta},(b_3)_{1,2}^{n-\theta},\cdots,(b_3)_{M_x-1,2}^{n-\theta},\\
&\qquad\qquad\qquad\cdots\cdots,(b_3)_{1,M_y-1}^{n-\theta},\cdots,(b_3)_{M_x-1,M_y-1}^{n-\theta}) ,\\
& {\bf u}^n:=(u_{1,1}^n,\cdots,u_{M_x-1,1}^n,u_{1,2}^n,\cdots,u_{M_x-1,2}^n,\cdots\cdots,u_{1,M_y-1}^n,\cdots,u_{M_x-1,M_y-1}^n)^T,
\end{align*}
where ${\bf P}_3^{n-\theta}$ and ${\bf P}_4^{n-\theta}$, with entries  coming from $(p_3)_{i,j}^{n-\theta}$ and $(p_4)_{i,j}^{n-\theta}$ respectively, are all $(M_x-1)(M_y-1)\times (M_x-1)(M_y-1)$ diagonal matrices defined similarly to ${\bf P}^{n-\theta}$, while $I_x$ and $I_y$ are $(M_x-1)$ and $(M_y-1)$ dimensional identity matrices respectively.
Furthermore, we have used the notations
{\small\begin{eqnarray}\nonumber
S_x:=\frac1{h_x}
\left[
\begin{array}{cccc}
-1  & & & \\
1& -1&  & \\
&   \ddots& \ddots& \\
&&  1& -1\\
&&& 1\\
\end{array}
\right]_{M_x\times(M_x-1)},~
{\hat S}_x:=\frac1{2h_x}
\left[
\begin{array}{ccccc}
-1  &1 & & &\\
& -1&  1& &\\
& &  \ddots& \ddots& \\
&& & -1& 1\\
&&&& -1\\
\end{array}
\right]_{(M_x-1)\times(M_x-1)},
\end{eqnarray}}
and $S_y$, ${\hat S}_y$ are defined in a similar way.

Therefore, if $p$ is non-increasing w.r.t. $t$ for every fixed ${\bf x}$, according to Lemma \ref{zQDz}, we have
\begin{align}\label{uPDu}
2({\bf u}^{n-\theta})^T{\bf P}^{n-\theta}({\cal D}_\tau^\beta {\bf u})^{n-\theta}\geq \sum_{k=1}^nA_{n-k}^{(n)}\nabla_\tau [({\bf u}^k)^T{\bf P}^{k-\theta}{\bf u}^k].
\end{align}
%with ${\bf P}^{-\theta}:={\bf P}^0$.

Moreover, taking
\begin{align}\label{ww}
{\bf u}_x^{n-\vartheta}:=(I_y\otimes S_x){\bf u}^{n-\vartheta},\quad {\bf u}_y^{n-\vartheta}:=(S_y\otimes I_x){\bf u}^{n-\vartheta},\quad \mbox{where}\quad \vartheta=\theta~\mbox{or}~0,
\end{align}
if $p_1$ and $p_2$ are all non-increasing w.r.t. $t$ for every fixed ${\bf x}$, we have
\begin{align}\nonumber
&2({\bf u}^{n-\theta})^T{\bf A}^{n-\theta}({\cal D}_{\tau}^\beta {\bf u})^{n-\theta}\\\nonumber
\geq&A_0^{(n)}({\bf u}^n)^T{\bf A}^{n-\theta}{\bf u}^n
-\sum_{k=1}^{n-1}(A_{n-k-1}^{(n)}-A_{n-k}^{(n)})({\bf u}^k)^T{\bf A}^{n-\theta}{\bf u}^k-A_{n-1}^{(n)}({\bf u}^0)^T{\bf A}^{n-\theta}{\bf u}^0\\\nonumber
=&A_0^{(n)}\left[({\bf u}_x^n)^T{\bf P}_1^{n-\theta}{\bf u}_x^n + ({\bf u}_y^n)^T{\bf P}_2^{n-\theta}{\bf u}_y^n\right]-\sum_{k=1}^{n-1}(A_{n-k-1}^{(n)}-A_{n-k}^{(n)})\left[({\bf u}_x^k)^T{\bf P}_1^{n-\theta}{\bf u}_x^k + ({\bf u}_y^k)^T{\bf P}_2^{n-\theta}{\bf u}_y^k\right]\\\nonumber
&-A_{n-1}^{(n)}\left[({\bf u}_x^0)^T{\bf P}_1^{n-\theta}{\bf u}_x^0+ ({\bf u}_y^0)^T{\bf P}_2^{n-\theta}{\bf u}_y^0\right]\\\nonumber
\geq & A_0^{(n)}\left[({\bf u}_x^n)^T{\bf P}_1^{n-\theta}{\bf u}_x^n + ({\bf u}_y^n)^T{\bf P}_2^{n-\theta}{\bf u}_y^n\right]-\sum_{k=1}^{n-1}(A_{n-k-1}^{(n)}-A_{n-k}^{(n)})\left[({\bf u}_x^k)^T{\bf P}_1^{k-\theta}{\bf u}_x^k + ({\bf u}_y^k)^T{\bf P}_2^{k-\theta}{\bf u}_y^k\right]\\\label{uADu}
&-A_{n-1}^{(n)}\left[({\bf u}_x^0)^T{\bf P}_1^{-\theta}{\bf u}_x^0+ ({\bf u}_y^0)^T{\bf P}_2^{-\theta}{\bf u}_y^0\right].
\end{align}
%with ${\bf P}_1^{-\theta}:={\bf P}_1^0$ and ${\bf P}_2^{-\theta}:={\bf P}_2^0$.

The two inequalities \eqref{uPDu} and \eqref{uADu} play critical roles in the analysis of  our methods. Therefore, we must fulfill the following:
\begin{itemize}
\item find a continuous and positive function $p({\bf x},t)$ which is non-increasing w.r.t. $t$ for every fixed ${\bf x}\in\Omega$ such that $p_1({\bf x},t)=pa_1$ and $p_2({\bf x},t)=pa_2$ are all non-increasing w.r.t. $t\in[0,T]$ for every fixed ${\bf x}\in\Omega$.
\end{itemize}
The above task can be completed by choosing the candidates presented in the following lemma.
\begin{lemma}\label{functionP}
For the positive variable coefficients $a_1$ and $a_2$, consider
\begin{align}\label{pp1p2}
p({\bf x},t):=\frac{d({\bf x})e^{-C_pt}}{a_1({\bf x},t)a_2({\bf x},t)},\quad p_1({\bf x},t):=\frac{d({\bf x})e^{-C_pt}}{a_2({\bf x},t)},\quad p_2({\bf x},t):=\frac{d({\bf x})e^{-C_pt}}{a_1({\bf x},t)},
\end{align}
for $ x\in\Omega,~t\in[0,T]$; where $C_p$ is the constant in {\bf V1} and $d({\bf x})$ is a positive and continuous function.
If  $a_1$ and $a_2$  satisfy  {\bf V1}, then the functions $p$, $p_1$ and $p_2$ are positive and continuous.
 Furthermore, they are all non-increasing w.r.t. $t$ for every fixed ${\bf x}\in\Omega$.
\end{lemma}
\begin{proof}
It is obvious that $p$, $p_1$ and $p_2$ are all positive and continues.

By taking the partial derivative w.r.t. $t$, we have
\begin{align*}
&p_t=d({\bf x})e^{-C_pt}\left[\frac{-C_pa_1a_2-(a_1a_2)_t}{(a_1a_2)^2}\right],\\
&(p_1)_t=d({\bf x})e^{-C_pt}\left[\frac{-C_pa_2-(a_2)_t}{(a_2)^2}\right],\\
&(p_2)_t=d({\bf x})e^{-C_pt}\left[\frac{-C_pa_1-(a_1)_t}{(a_1)^2}\right].
\end{align*}
Then it is easy to reach the desired result provided the assumptions in {\bf V1} hold.
\end{proof}
\begin{remark}
Since the numerical methods proposed later depend on precise choices of $p$, $p_1$ and $p_2$, here we list some simple candidates for  $d({\bf x})$ and $C_p$. In fact, the pool for choices is large. One may take $d({\bf x})=1,e^{\sin(x+y)},e^{\cos(x+y)}$ and $C_p=\sup\{\left|(a_1)_t/a_1\right|+\left|(a_2)_t/a_2\right|\}$, etc.
\end{remark}

In the rest of this paper, we always take functions $p$, $p_1$ and $p_2$ as those given in \eqref{pp1p2}.
 Consequently the two inequalities \eqref{uPDu} and \eqref{uADu} are true basing on {\bf V1}.
 %\textcolor{blue}{(??)We remark that the assumptions on the variable coefficients will not be stronger than {\bf V1} and {\bf V2} in the whole paper.}

\section{The sub-diffusion equation with time-space dependent coefficients}\label{sub-diffusion}

\subsection{The numerical scheme}

Let $u_h^k$ be the numerical approximations of $u({\bf x}_h,t_k)$, ${\bf x}_h\in \Omega_h,0\leq k\leq N$.
Denote $u_h^{n-\theta}:=(1-\theta)u_h^n+\theta u_h^{n-1}$, $f_h^{n-\theta}:=f({\bf x}_h,t_{n-\theta})$ for $n\geq 1$ and $\varphi_h:=\varphi({\bf x}_h)$.

From Section \ref{new-technique} (noting that $\beta=\alpha$ here), it is natural to construct an implicit scheme to solve the sub-diffusion problem \eqref{eq1}--\eqref{eq2} in the following form:
\begin{align}\label{sub-sc1}
&({\cal D}_{\tau}^\alpha u_h)^{n-\theta}={\cal A}_h^{n-\theta} u_h^{n-\theta}+f_h^{n-\theta}, \quad  {\bf x}_h\in \Omega_h, 1\leq n\leq N;\\\label{sub-sc2}
& u_h^0=\varphi_h, \quad  {\bf x}_h\in \Omega_h,
\end{align}
subject to the zero boundary conditions.

%\begin{remark}
%The equations \eqref{sub-sc1}--\eqref{sub-sc2} represent two different numerical schemes to solve the sub-diffusion equation \eqref{eq1}--\eqref{eq2}. It is the nonuniform L1 scheme while $\theta=0$ and is the nonuniform Alikhanov scheme while $\theta=\alpha/2$.
%\end{remark}

To perform the numerical analysis, we rewrite the scheme \eqref{sub-sc1}--\eqref{sub-sc2} in  the following  matrix representation:
\begin{align}\label{sub-matrix-form1}
&({\cal D}_{\tau}^\alpha {\bf u})^{n-\theta}=\left[-({\bf P}^{n-\theta})^{-1}{\bf A}^{n-\theta}+{\bf B}^{n-\theta}+{\bf C}^{n-\theta}\right]{\bf u}^{n-\theta}+{\bf f}^{n-\theta}, \quad 1\leq n\leq N;\\\label{sub-matrix-form2}
& {\bf u}^0= \varPhi;
\end{align}
where ${\bf u}^{n-\theta}:=(1-\theta){\bf u}^{n}+\theta {\bf u}^{n-1}$ and
\begin{align*}
&{\bf f}^{n-\theta}:=(f_{1,1}^{n-\theta},\cdots,f_{M_x-1,1}^{n-\theta},f_{1,2}^{n-\theta},\cdots,f_{M_x-1,2}^{n-\theta},\cdots\cdots,f_{1,M_y-1}^{n-\theta},\cdots,f_{M_x-1,M_y-1}^{n-\theta})^T,\\
& \varPhi:=(\varphi_{1,1},\cdots,\varphi_{M_x-1,1}\,\varphi_{1,2}, \cdots,\varphi_{M_x-1,2},\cdots\cdots,\varphi_{1,M_y-1},\cdots,\varphi_{M_x-1,M_y-1})^T.
\end{align*}

\subsection{Stability and convergence}

The next lemma shows a discrete fractional Gr\"{o}nwall inequality which is a slightly modified version of \cite[Theorem 3.1]{LiaoGronwall} (noting that $\pi_A=11/4$ and $\rho$ is the maximum time-step ratio (see  {Appendix})).
\begin{lemma}\cite[Lemma 3.2]{LyuVong2020diffu-wave}\label{gronwall-lemma1}
Let $(g^n)_{n=1}^N$ and $(\lambda_l)_{l=0}^{N-1}$ be given nonnegative sequences. Assume that there exists a constant $\Lambda$ (independent of the step sizes) such that $\Lambda\geq\sum_{l=0}^{N-1}\lambda_l$, and that the maximum step size satisfies
$$\max_{1\leq n\leq N}\tau_n\leq \frac1{{^\alpha\sqrt{4\pi_A\Gamma(2-\alpha)\Lambda}}}.$$
Then, for any nonnegative sequences $(u^k)_{k=0}^N$ and $(v^k)_{k=0}^N$  satisfying
\begin{align*}%\label{gronwall1}
\sum_{k=1}^nA_{n-k}^{(n)}\nabla_{\tau} \left[(u^k)^2+(v^k)^2 \right] \leq& \sum_{k=1}^n \lambda_{n-k}\left(u^{k-\theta}+v^{k-\theta}\right)^2+ (u^{n-\theta}+v^{n-\theta})g^{n}, \quad  1\leq n\leq N,
\end{align*}
it holds that
\begin{align}\label{gronwall2}
u^n+v^n\leq  4E_{\alpha}(4\max(1,\rho)\pi_A\Lambda t_n^\alpha)\left(u^0+v^0+\max_{1\leq k\leq n}\sum_{j=1}^kP_{k-j}^{(k)}g^{j} \right)\quad \mbox{for}~1\leq n\leq N,
\end{align}
where $E_{\alpha}(z)=\sum_{k=0}^\infty\frac{z^k}{\Gamma(1+k\alpha)}$ is the Mittag-Leffler function.
\end{lemma}
The coefficients $P_{n-j}^{(n)}$ in \eqref{gronwall2} are called the discrete complementary convolution kernels (see more details in \cite{LiaoGronwall}), and they satisfy (\cite[Lemmma 2.1]{LiaoGronwall})
\begin{align}\label{P-proper}
0\leq P_{n-j}^{(n)}\leq \pi_A\Gamma(2-\alpha)\tau_j^\alpha,\quad \sum_{j=1}^nP_{n-j}^{(n)}\omega_{1-\alpha}(t_j)\leq\pi_A,\quad 1\leq j\leq n\leq N.
\end{align}

 For $u_h,v_h$ belonging to the space of grid functions which vanish on $\partial\Omega_h$, we introduce the discrete inner product $\langle u,v \rangle:=h_xh_y\sum_{{\bf x}_h\in\Omega_h}u_hv_h$, the discrete $L^2$-norm
  $ \|u\|:=\sqrt{\langle u,u \rangle}$, %the discrete $L_\infty$-norm $\|u\|_{\infty}:=\max\{|u_h|\}$,
  the discrete $H^1$ seminorms $\|\delta_x u\|$ and $\|\delta_y u\|$, and $\|\nabla_h u\|:=\sqrt{\|\delta_x u\|^2+\|\delta_y u\|^2}$. Suppose ${\tilde C}_0$, ${\hat C}_0$, ${\tilde C}_l$ and ${\hat C}_l$ are positive constants such that
$${\tilde C}_0 \leq |p({\bf x},t)|\leq {\hat C}_0, \quad {\tilde C}_l \leq |p_l({\bf x},t)|\leq {\hat C}_l\quad \mbox{for} \quad l=1,2,3,4.$$
Now we are going to show the stability and convergence for the proposed scheme \eqref{sub-sc1}--\eqref{sub-sc2}.
\begin{theorem}\label{Stability-sub}
If {\bf V1} is valid, the numerical scheme \eqref{sub-sc1}--\eqref{sub-sc2} is unconditionally stable and the discrete solutions $u_h^n~({\bf x}_h\in \Omega_h,1\leq n\leq N)$  satisfy
\begin{align*}
\|\nabla_h u^n\|
\leq& C \left(\|\nabla_h u^0\|+\max_{1\leq k\leq n}\sum_{j=1}^k P_{k-j}^{(k)} \|\nabla_h f^{j-\theta}\|\right)
\leq C \left(\|\nabla_h u^0\|+\max_{1\leq k\leq n}\{t_k^\alpha \|\nabla_h f^{k-\theta}\}\|\right).
\end{align*}
\end{theorem}
\begin{proof}
Multiplying  both sides of \eqref{sub-matrix-form1} by $({\bf A}^{n-\theta}{\bf u}^{n-\theta})^T$ gives:
\begin{align}\nonumber
&({\bf u}^{n-\theta})^T{\bf A}^{n-\theta}({\cal D}_{\tau}^\alpha {\bf u})^{n-\theta}+({\bf A}^{n-\theta}{\bf u}^{n-\theta})^T({\bf P}^{n-\theta})^{-1}({\bf A}^{n-\theta}{\bf u}^{n-\theta})\\\label{stb1}
=&({\bf u}^{n-\theta})^T\left[{\bf A}^{n-\theta}{\bf B}^{n-\theta}+{\bf A}^{n-\theta}{\bf C}^{n-\theta}\right]{\bf u}^{n-\theta}+({\bf u}^{n-\theta})^T{\bf A}^{n-\theta}{\bf f}^{n-\theta},\quad 1\leq n\leq N.
\end{align}
The first term on the left-hand side of \eqref{stb1} is evaluated by  \eqref{uADu}.

For the terms on the right-hand side, we first notice that for a real vector ${\bf z}=(z_1,z_2,\ldots,z_{M_x-1})^T$,
\begin{align*}
&4h_x^2({\hat S}_x{\bf z})^T({\hat S}_x{\bf z})=\sum_{i=1}^{M_x-2}(u_{i+1}-u_i)^2+u_{M_x-1}^2\leq
u_1^2+\sum_{i=1}^{M_x-2}(u_{i+1}-u_i)^2+u_{M_x-1}^2=h_x^2(S_x{\bf z})^T(S_x{\bf z}),\\
&4h_x^2({\hat S}_x^T{\bf z})^T({\hat S}_x^T{\bf z})=u_1^2+\sum_{i=1}^{M_x-2}(u_{i+1}-u_i)^2\leq h_x^2(S_x{\bf z})^T(S_x{\bf z}).
\end{align*}
Then it further holds that
\begin{align*}%\label{Sux}
&\max\left\{[(I_y\otimes {\hat S}_x){\bf u}^k]^T[(I_y\otimes {\hat S}_x){\bf u}^k],[(I_y\otimes {\hat S}_x^T){\bf u}^k]^T[(I_y\otimes {\hat S}_x^T){\bf u}^k]\right\}
\leq \frac14[(I_y\otimes  S_x){\bf u}^k]^T[(I_y\otimes S_x){\bf u}^k],\\%\label{Suy}
&\max\left\{[({\hat S}_y\otimes I_x){\bf u}^k]^T[({\hat S}_y\otimes I_x){\bf u}^k],[( {\hat S}_y^T\otimes I_x){\bf u}^k]^T[({\hat S}_y^T\otimes I_x){\bf u}^k]\right\}
\leq \frac14[(S_y\otimes I_x){\bf u}^k]^T[(S_y\otimes I_x){\bf u}^k].
\end{align*}
Thus the Cauchy-Schwarz inequality leads to
\begin{align}\nonumber
&({\bf u}^{n-\theta})^T\left[({\bf B}^{n-\theta})^T{\bf B}^{n-\theta}\right]{\bf u}^{n-\theta}\\\nonumber
\leq& \left[\left(I_y\otimes({\hat S}_x-{\hat S}_x^T)\right){{\bf u}^{n-\theta}} \right]^T({\bf P}_3^{n-\theta})^2\left[\left(I_y\otimes({\hat S}_x-{\hat S}_x^T)\right){{\bf u}^{n-\theta}} \right]\\\nonumber
&+\left[\left(({\hat S}_y-{\hat S}_y^T)\otimes I_x\right){\bf u}^{n-\theta} \right]^T({\bf P}_4^{n-\theta})^2\left[\left(({\hat S}_y-{\hat S}_y^T)\otimes I_x\right){\bf u}^{n-\theta} \right]\\\nonumber
\leq& 2{\hat C}_3^2\left\{\left[(I_y\otimes{\hat S}_x){\bf u}^{n-\theta} \right]^T \left[(I_y\otimes{\hat S}_x){\bf u}^{n-\theta} \right]+\left[(I_y\otimes{\hat S}_x^T){\bf u}^{n-\theta} \right]^T \left[(I_y\otimes{\hat S}_x^T){\bf u}^{n-\theta} \right] \right\}\\\nonumber
&+2{\hat C}_4^2\left\{\left[({\hat S}_y\otimes I_x){\bf u}^{n-\theta} \right]^T \left[({\hat S}_y\otimes I_x){\bf u}^{n-\theta} \right]+\left[({\hat S}_y^T\otimes I_x){\bf u}^{n-\theta} \right]^T \left[({\hat S}_y^T\otimes I_x){\bf u}^{n-\theta} \right]\right\}\\\nonumber
\leq& {\hat C}_3^2 ({\bf u}_x^{n-\theta})^T{\bf u}_x^{n-\theta}+{\hat C}_4^2 ({\bf u}_y^{n-\theta})^T{\bf u}_y^{n-\theta}\\\label{stab2-1}
\leq& C_5\left[ ({\bf u}_x^{n-\theta})^T{\bf P}_1^{n-\theta}{\bf u}_x^{n-\theta}+({\bf u}_y^{n-\theta})^T{\bf P}_2^{n-\theta}{\bf u}_y^{n-\theta} \right],
\end{align}
where $C_5:=\max\{{\hat C}_3^2,{\hat C}_4^2\}\cdot\max\{1/{\tilde C}_1,1/{\tilde C}_2\}$. Then the first part of the first term on the right-hand side of \eqref{stb1} can be estimated as
\begin{align}\nonumber
&2({\bf u}^{n-\theta})^T{\bf A}^{n-\theta}{\bf B}^{n-\theta}{\bf u}^{n-\theta}\\\nonumber
\leq&\frac{1}{{\hat C}_0}({\bf A}^{n-\theta}{\bf u}^{n-\theta})^T({\bf A}^{n-\theta}{\bf u}^{n-\theta})+{\hat C}_0({\bf u}^{n-\theta})^T\left[({\bf B}^{n-\theta})^T{\bf B}^{n-\theta}\right]{\bf u}^{n-\theta}\\\label{stb2}
\leq& ({\bf A}^{n-\theta}{\bf u}^{n-\theta})^T({\bf P}^{n-\theta})^{-1}({\bf A}^{n-\theta}{\bf u}^{n-\theta})+
{\hat C}_0C_5\left[ ({\bf u}_x^{n-\theta})^T{\bf P}_1^{n-\theta}{\bf u}_x^{n-\theta}+({\bf u}_y^{n-\theta})^T{\bf P}_2^{n-\theta}{\bf u}_y^{n-\theta} \right].
\end{align}
Noticing the embedding inequality $\|u^k\|\leq C_{\Omega} \|\nabla u^k\|$, $k\geq 0$, it leads to
\begin{align}\label{embedding-ine}
({\bf u}^{n-\theta})^T{\bf u}^{n-\theta}\leq C_{\Omega}^2 \left[ ({\bf u}_x^{n-\theta})^T{\bf u}_x^{n-\theta}+({\bf u}_y^{n-\theta})^T{\bf u}_y^{n-\theta}\right].
\end{align}
Then similar to the derivation of \eqref{stb2}, one gets
\begin{align}\nonumber
&2({\bf u}^{n-\theta})^T{\bf A}^{n-\theta}{\bf C}^{n-\theta}{\bf u}^{n-\theta}\\\nonumber
\leq&\frac{1}{{\hat C}_0}({\bf A}^{n-\theta}{\bf u}^{n-\theta})^T({\bf A}^{n-\theta}{\bf u}^{n-\theta})+{\hat C}_0({\bf u}^{n-\theta})^T\left[({\bf C}^{n-\theta})^T{\bf C}^{n-\theta}\right]{\bf u}^{n-\theta}\\\nonumber
\leq& ({\bf A}^{n-\theta}{\bf u}^{n-\theta})^T({\bf P}^{n-\theta})^{-1}({\bf A}^{n-\theta}{\bf u}^{n-\theta})+
{\hat C}_0C_3^2({\bf u}^{n-\theta})^T{\bf u}^{n-\theta}\\\label{stb3}
\leq& ({\bf A}^{n-\theta}{\bf u}^{n-\theta})^T({\bf P}^{n-\theta})^{-1}({\bf A}^{n-\theta}{\bf u}^{n-\theta})+
{\hat C}_0C_6\left[ ({\bf u}_x^{n-\theta})^T{\bf P}_1^{n-\theta}{\bf u}_x^{n-\theta}+({\bf u}_y^{n-\theta})^T{\bf P}_2^{n-\theta}{\bf u}_y^{n-\theta} \right],
\end{align}
where $C_6:=\max\{1/{\tilde C}_1,1/{\tilde C}_2\}C_3^2C_{\Omega}^2$.

For the last term on the right-hand side of \eqref{stb1}, we have
\begin{align}\nonumber
({\bf u}^{n-\theta})^T{\bf A}^{n-\theta}{\bf f}^{n-\theta}
=& \left[(I_y\otimes S_x){\bf u}^{n-\theta}\right]^T{\bf P}_1^{n-\theta}(I_y\otimes S_x){\bf f}^{n-\theta}\\\nonumber
&+\left[(S_y\otimes I_x){\bf u}^{n-\theta}\right]^T{\bf P}_2^{n-\theta}(S_y\otimes I_x){\bf f}^{n-\theta}\\\nonumber
=& ({\bf u}_x^{n-\theta})^T{\bf P}_1^{n-\theta}(I_y\otimes S_x){\bf f}^{n-\theta}
+({\bf u}_y^{n-\theta})^T{\bf P}_2^{n-\theta}(S_y\otimes I_x){\bf f}^{n-\theta}\\\nonumber
\leq& \sqrt{{\hat C}_1} \sqrt{({\bf u}_x^{n-\theta})^T{\bf P}_1^{n-\theta}{\bf u}_x^{n-\theta}}\sqrt{\left[ (I\otimes S_x){\bf f}^{n-\theta}\right]^T(I\otimes S_x){\bf f}^{n-\theta}}\\\label{stb4}
&+\sqrt{{\hat C}_2} \sqrt{({\bf u}_y^{n-\theta})^T{\bf P}_2^{n-\theta}{\bf u}_y^{n-\theta}}\sqrt{\left[ (S_y\otimes I){\bf f}^{n-\theta}\right]^T(S_y\otimes I){\bf f}^{n-\theta}}.
\end{align}
Therefore, it follows from \eqref{stb1}--\eqref{stb4} and \eqref{uADu} that
\begin{align}\nonumber
&\sum_{k=1}^nA_{n-k}^{(n)}\nabla_\tau \left[({\bf u}_x^k)^T{\bf P}_1^{k-\theta}{\bf u}_x^k + ({\bf u}_y^k)^T{\bf P}_2^{k-\theta}{\bf u}_y^k\right]\\\nonumber
\leq&{\hat C}_0(C_5+C_6)\left[ ({\bf u}_x^{n-\theta})^T{\bf P}_1^{n-\theta}({\bf u}_x^{n-\theta})+({\bf u}_y^{n-\theta})^T{\bf P}_2^{n-\theta}({\bf u}_y^{n-\theta}) \right]\\\nonumber
& +2\sqrt{{\hat C}_1} \sqrt{({\bf u}_x^{n-\theta})^T{\bf P}_1^{n-\theta}{\bf u}_x^{n-\theta}}\sqrt{\left[ (I_y\otimes S){\bf f}^{n-\theta}\right]^T(I_y\otimes S){\bf f}^{n-\theta}}\\\label{stb-x}
&+2\sqrt{{\hat C}_2} \sqrt{({\bf u}_y^{n-\theta})^T{\bf P}_2^{n-\theta}{\bf u}_y^{n-\theta}}\sqrt{\left[ (S\otimes I_x){\bf f}^{n-\theta}\right]^T(S\otimes I_x){\bf f}^{n-\theta}}.
\end{align}
In view of the relationships
$$\|\delta_x u^{k-\theta}\|=\sqrt{h_xh_y({\bf u}_x^{k-\theta})^T({\bf u}_x^{k-\theta})}\quad \mbox{and}\quad \|\delta_y u^{k-\theta}\|=\sqrt{h_xh_y({\bf u}_y^{k-\theta})^T({\bf u}_y^{k-\theta})},$$ where $k\geq 0$, we define the following norms
$$\|\delta_x u^{k-\theta}\|_{P_1}:=\sqrt{h_xh_y({\bf u}_x^{k-\theta})^T{\bf P}_1^{k-\theta}({\bf u}_x^{k-\theta})}\quad \mbox{and}\quad
\|\delta_y u^{k-\theta}\|_{P_2}:=\sqrt{h_xh_y({\bf u}_y^{k-\theta})^T{\bf P}_2^{k-\theta}({\bf u}_y^{k-\theta})}.$$
Moreover, denote
$$\|v\|_{P_k}^{(n-\theta)}:=(1-\theta)\|v^n\|_{P_k}+\theta\|v^{n-1}\|_{P_k}  \quad \mbox{for}~v_h\in\Omega_h~\mbox{and}~k=1,2.$$
Then the triangle inequality yields $\|v^{n-\theta}\|_{P_k}\leq\|v^{(n-\theta)}\|_{P_k}$.

Now, take $C_7:=2\max\{\sqrt{{\hat C}_1} ,\sqrt{{\hat C}_2} \}$. Multiplying  both sides of the inequality \eqref{stb-x} by $h_xh_y$, it follows
\begin{align*}
&\sum_{k=1}^nA_{n-k}^{(n)}\nabla_\tau \left[\|\delta_x u^k\|_{P_1}^2 + \|\delta_y u^k\|_{P_1}^1 \right]\\
\leq&{\hat C}_0(C_5+C_6)\left[ \|\delta_x u^{n-\theta}\|_{P_1}^2+ \|\delta_y u^{n-\theta}\|_{P_2}^2 \right] +2\sqrt{{\hat C}_1} \|\delta_x u^{n-\theta}\|_{P_1}\cdot\|\delta_x f^{n-\theta}\|\\
&+2\sqrt{{\hat C}_2} \|\delta_y u^{n-\theta}\|_{P_2}\cdot\|\delta_y f^{n-\theta}\|\\
\leq&{\hat C}_0(C_5+C_6)\left[ \|\delta_x u^{n-\theta}\|_{P_1}^2+ \|\delta_y u^{n-\theta}\|_{P_2}^2 \right] +C_7\left( \|\delta_x u^{n-\theta}\|_{P_1}+ \|\delta_y u^{n-\theta}\|_{P_2} \right) \|\nabla_h f^{n-\theta}\|\\
\leq& {\hat C}_0(C_5+C_6)\left[ \left(\|\delta_x u\|_{P_1}^{(n-\theta)}\right)^2+\left(\|\delta_y u\|_{P_2}^{(n-\theta)}\right)^2\right]+C_7 \left( \|\delta_x u\|_{P_1}^{(n-\theta)}+ \|\delta_y u\|_{P_2}^{(n-\theta)} \right) \|\nabla_h f^{n-\theta}\|.
\end{align*}
Applying Lemma \ref{gronwall-lemma1}, we get
\begin{align*}
&\|\delta_x u^n\|_{P_1} + \|\delta_y u^n\|_{P_2}\\
\leq& 4E_{\beta}\left(4\max(1,\rho)\pi_A {\hat C}_0(C_5+C_6) t_n^\beta\right) \left[\|\delta_x u^0\|_{P_1} + \|\delta_y u^0\|_{P_2}+C_7\max_{1\leq k\leq n}\sum_{j=1}^k P_{k-j}^{(k)} \|\nabla_h f^{k-\theta}\|\right].
\end{align*}
Since $\|\delta_x u^0\|_{P_1}\leq \sqrt{{\hat C}_1}\|\delta_x u^0\|$, $\|\delta_y u^0\|_{P_2}\leq \sqrt{{\hat C}_2}\|\delta_y u^0\|$, and
\begin{align*}
\|\delta_x u^n\|=\sqrt{h_xh_y({\bf u}_x^n)^T{\bf u}_x^n} \leq \frac1{\sqrt{{\tilde C}_1}}\|\delta_x u^n\|_{P_1},
\quad
\|\delta_y u^n\|=\sqrt{h_xh_y({\bf u}_y^n)^T{\bf u}_y^n} \leq \frac1{\sqrt{{\tilde C}_2}}\|\delta_y u^n\|_{P_2},
\end{align*}
we obtain
\begin{align*}
\|\nabla_h u^n\|\leq \|\delta_x u^n\|+\|\delta_y u^n\|
\leq& \max\{\frac1{\sqrt{{\tilde C}_1}},\frac1{\sqrt{{\tilde C}_2}}\}\cdot\left( \|\delta_x u^n\|_{P_1} + \|\delta_y u^n\|_{P_2}\right)\\
\leq& C \left(\|\nabla_h u^0\|+\max_{1\leq k\leq n}\sum_{j=1}^k P_{k-j}^{(k)} \|\nabla_h f^{j-\theta}\|\right)\\
\leq& C \left(\|\nabla_h u^0\|+\max_{1\leq k\leq n}\{t_k^\alpha \|\nabla_h f^{k-\theta}\|\}\right),
\end{align*}
where \eqref{P-proper} has been utilized.
\end{proof}

\begin{remark}
We remark that one may consider numerical approximations of \eqref{eq1}  based on a more simplified equivalent equation with
\begin{align}\label{eq-alternative}
{\cal A}u=\partial_x(a_1\partial_xu)+\partial_y(a_2\partial_yu)+[b_1-(a_1)_x]\partial_xu+[b_2-(a_2)_y]\partial_yu+b_3u.
\end{align}
The corresponding numerical approximation of \eqref{eq-alternative} will be
\begin{align}\label{sub-sc-alternative}
&({\cal D}_{\tau}^\alpha {\bf u})^{n-\theta}=\left[-{\tilde {\bf A}}^{n-\theta}+{\tilde{\bf B}}^{n-\theta}+{\tilde {\bf C}}^{n-\theta}\right]{\bf u}^{n-\theta}+{\bf f}^{n-\theta}, \quad 1\leq n\leq N;
\end{align}
where ${\tilde {\bf A}}^{n-\theta}:=(I_y\otimes S_x)^T{\bf A}_1^{n-\theta}(I_y\otimes S_x)+(S_y\otimes I_x)^T{\bf A}_2^{n-\theta}(S_y\otimes I_x)$,  and ${\bf A}_1$, ${\bf A}_2$, ${\tilde {\bf B}}$ and ${\tilde {\bf C}}$ are  diagonal matrices with entries from corresponding variable coefficients in \eqref{eq-alternative}.

To obtain the unconditional $H^1$-norm stability and convergence, one should multiply both sides of \eqref{sub-sc-alternative} by $({\tilde {\bf A}}^{n-\theta}{\bf u}^{n-\theta})^T$, which leads to a serious difficulty for   estimating the term\\ $({\bf u}^{n-\theta})^T{\tilde {\bf A}}^{n-\theta}({\cal D}_{\tau}^\alpha {\bf u})^{n-\theta}$ on the left-hand side. This is the main reason why we introduce the concise technique in Section \ref{new-technique}. The advantage of such technique will be more obvious for diffusion-wave equation as its numerical approximations have a coupled structure {\rm (}see also{\rm \eqref{dw-matrix-form-1}--\eqref{dw-matrix-form-2})}. For more details, see the first three steps \eqref{stb1-DW}--\eqref{stb3-DW} of the proof in the next section.
\end{remark}

Next, we show the convergence of the proposed scheme \eqref{sub-matrix-form1}--\eqref{sub-matrix-form2}.
\begin{theorem}\label{Convergence-sub}
Denote $e_h^k:=u({\bf x}_h,t_k)-u_h^k~({{\bf x}_h}\in\bar\Omega_h,~0\leq k\leq N)$. If {\bf V1}, {\bf V2}, {\bf MA} and the assumptions in \eqref{regularity1}--\eqref{regularity2} are valid, the numerical scheme \eqref{sub-sc1}--\eqref{sub-sc2} is unconditionally convergent with
\begin{equation}
\|\nabla_h e^n\|\leq
C(\tau^{\{2,\gamma\sigma_1\}}+h_x^2+h_y^2),
\quad \mbox{for}\quad 1\leq n\leq N.
\end{equation}
\end{theorem}

\begin{proof}
Denote ${\bf e}^k$ the error vector with enteries $e_{i,j}^k$ being arranged similar to those of ${\bf u}^k$. One can easily obtain the error equations
\begin{align}\label{sub-error1}
&({\cal D}_{\tau}^\alpha {\bf e})^{n-\theta}=\left[-({\bf P}^{n-\theta})^{-1}{\bf A}^{n-\theta}+{\bf B}^{n-\theta}+{\bf C}^{n-\theta}\right]{\bf e}^{n-\theta}+{\bf R}^{n-\theta}, \quad  1\leq n\leq N;\\\label{sub-error2}
& {\bf e}^0={\bf 0},
\end{align}
where ${\bf e}^{n-\theta}:=(1-\theta){\bf e}^{n}+\theta{\bf e}^{n-1}$, and
\begin{align*}
%&{\bf e}^k:=(e_{1,1}^k,\cdots,e_{M_x-1,1}^k,e_{1,2}^k,\cdots,e_{M_x-1,2}^k,\cdots\cdots,e_{1,M_y-1}^k,\cdots,e_{M_x-1,M_y-1}^k),\quad 1\leq k\leq N,\\
{\bf R}^{n-\theta}:=(R_{1,1}^{n-\theta},\cdots,R_{M_x-1,1}^{n-\theta},R_{1,2}^{n-\theta},\cdots,R_{M_x-1,2}^{n-\theta},\cdots\cdots,R_{1,M_y-1}^{n-\theta},\cdots,R_{M_x-1,M_y-1}^{n-\theta})
\end{align*}
with
 \begin{align}\label{TE-sub}
 R_h^{n-\theta}=-{\cal T}_u(x_h,t_{n-\theta})+{\cal T}_A(x_h,t_{n-\theta})+{\cal S}(x_h,t_{n-\theta}),\quad {\bf x}_h\in \Omega_h.
 \end{align}
The estimation of the temporal and spatial truncation errors ${\cal T}_u(x_h,t_{n-\theta})$, ${\cal T}_A(x_h,t_{n-\theta})$ and ${\cal S}(x_h,t_{n-\theta})$ are given in the Appendix (Subsection \ref{TEA}).

Following the proof of Theorem \ref{Stability-sub}, we can get
 \begin{align}\label{sub-con-1}
\|\nabla_h e^n\|
\leq C \max_{1\leq k\leq n}\sum_{j=1}^k P_{k-j}^{(k)} \|\nabla_h R^{j-\theta}\|,\quad 1\leq n\leq N.
\end{align}
 Therefore, the claimed result can be verified by combining \eqref{sub-con-1}, \eqref{nab_S} and \eqref{nab_TA}--\eqref{nab_Tu}.
\end{proof}

\section{The diffusion-wave equation with time-space dependent coefficients}\label{diffusion-wave}

\subsection{The numerical scheme}
A novel order reduction method (SFOR) proposed in \cite{LyuVong2020diffu-wave} will be employed to construct efficient numerical scheme on nonuniform time partitions for the diffusion-wave problem \eqref{eq1} and \eqref{eq3}. The underlying idea of the SFOR method is demonstrated in the following lemma.

\begin{lemma}\cite[Lemma 2.1]{LyuVong2020diffu-wave}\label{sym-reduction}
For $\alpha\in(1,2)$ and $u(t)\in {\cal C}^2((0,T])$, it holds that
\begin{align*}%\label{reduction1}
{\cal D}_t^{\alpha}u(t)={\cal D}_t^{\frac{\alpha}{2}}\left({\cal D}_t^{\frac{\alpha}{2}} u(t)\right)-u'(0)\omega_{2-\alpha}(t).
\end{align*}
Moreover, if we take ${\tilde u}(t):=u(t)-tu'(0)$, then
\begin{align*}%\label{reduction2}
{\cal D}_t^\alpha {u}(t)={\cal D}_t^\alpha {\tilde u}(t)={\cal D}_t^{\frac{\alpha}{2}}\left({\cal D}_t^{\frac{\alpha}{2}} {\tilde u}(t)\right).
\end{align*}
\end{lemma}
Utilizing Lemma \ref{sym-reduction}, the equation \eqref{eq1} can be rewritten as ($\beta=\alpha/2$ here)
\begin{align}\label{eq4}
&{\cal D}_t^\beta v={\cal A}{\tilde u} +f({\bf x},t)+{\cal A}(t\psi),\\\label{eq4-2}
&v={\cal D}_t^\beta {\tilde u},
\end{align}
with $ {\tilde u}=u-t\psi$, for  ${\bf x}\in\Omega$ and $t\in(0,T]$.

It is obvious that the problem \eqref{eq4}--\eqref{eq4-2} is equivalently to %\eqref{eq4}--\eqref{eq4-2}  as well as
\eqref{eq1} and \eqref{eq3} provided $u(\cdot,t)\in{\cal C}^2((0,T])$ and $p$ is invertible, i.e., they have the  same analytical solution. Then we can design the numerical approximation based on the model \eqref{eq4}--\eqref{eq4-2} in order to solve the original problem \eqref{eq1} and \eqref{eq3}.

By using the discrete Caputo formula \eqref{dis-Caputo} and the discrete operator ${\cal A}_h^{n-\theta}$ given in Section \ref{new-technique}, with ${\tilde u}_h^n=u_h^n-t_n\psi_h$, we propose the following implicit numerical scheme for solving \eqref{eq4}--\eqref{eq4-2}:
\begin{align}\label{dw-sc1}
&({\cal D}_{\tau}^\beta v_h)^{n-\theta}={\cal A}_h^{n-\theta} {\tilde u}_h^{n-\theta}+f_h^{n-\theta}+[{\cal A}(t\psi)]_h^{n-\theta}, \quad  {\bf x}_h\in \Omega_h, 1\leq n\leq N;\\\label{dw-sc2}
& v_h^{n-\theta}=({\cal D}_{\tau}^\beta {\tilde u}_h)^{n-\theta},\quad  {\bf x}_h\in \Omega_h, 1\leq n\leq N;
\end{align}
 subject to the zero boundary conditions and initial conditions $u_h^0=\phi_h$ and $v_h^0=0$.

%One should notice that the equations \eqref{dw-sc1}--\eqref{dw-sc2}  stand for two different numerical schemes, i.e., it is the nonuniform L1 scheme while $\theta=0$ and the Alikhanov scheme while $\theta=\beta/2=\alpha/4$.

Denote ${\tilde \psi}_h^{n-\theta}:=[{\cal A}(t\psi)]_h^{n-\theta}$, and
\begin{align*}
{\bf \Psi}^{n-\theta}:=\diag\left({\tilde \psi}^{n-\theta}_{1,1},\cdots,{\tilde \psi}^{n-\theta}_{M_x-1,1},{\tilde \psi}^{n-\theta}_{1,2},\cdots,{\tilde \psi}^{n-\theta}_{M_x-1,2},\cdots\cdots,{\tilde \psi}^{n-\theta}_{1,M_y-1},\cdots,{\tilde \psi}^{n-\theta}_{M_x-1,M_y-1}\right).
\end{align*}
The matrix form of the numerical scheme \eqref{dw-sc1}--\eqref{dw-sc2} is:
\begin{align}\label{dw-matrix-form-1}
&({\cal D}_{\tau}^\beta {\bf v})^{n-\theta}=\left[-({\bf P}^{n-\theta})^{-1}{\bf A}^{n-\theta}+{\bf B}^{n-\theta}+{\bf C}^{n-\theta}\right]{\tilde{\bf u}}^{n-\theta}+{\bf f}^{n-\theta}+{\bf \Psi}^{n-\theta};\\\label{dw-matrix-form-2}
& {\bf v}^{n-\theta}=({\cal D}_{\tau}^\beta {\tilde{\bf u}})^{n-\theta};
\end{align}
for ${\bf x}_h\in \Omega_h,~ 1\leq n\leq N$.

\subsection{Stability and convergence}
In the same way as Lemma \ref{gronwall-lemma1}, we can also simply go through the proof of \cite[Theorem 3.1]{LiaoGronwall} to have an analogy version of the discrete fractional Gr\"{o}nwall inequality (with $\pi_A=11/4$):
\begin{lemma}\label{gronwall-lemma2}
Let $(g^n)_{n=1}^N$ and $(\lambda_l)_{l=0}^{N-1}$ be given nonnegative sequences. Assume that there exists a constant $\Lambda$ (independent of the step sizes) such that $\Lambda\geq\sum_{l=0}^{N-1}\lambda_l$, and that the maximum step size satisfies
$$\max_{1\leq n\leq N}\tau_n\leq \frac1{{^\beta\sqrt{4\pi_A\Gamma(2-\beta)\Lambda}}}.$$
Then, for any nonnegative sequence $(u^k)_{k=0}^N$, $(v^k)_{k=0}^N$ and $(w^k)_{k=0}^N$  satisfying
\begin{align*}%\label{gronwall1}
\sum_{k=1}^nA_{n-k}^{(n)}\nabla_{\tau} \left[(u^k)^2+(v^k)^2+(w^k)^2 \right] \leq& \sum_{k=1}^n \lambda_{n-k}\left(u^{k-\theta}+v^{k-\theta}+w^{k-\theta}\right)^2\\
&+ (u^{n-\theta}+v^{n-\theta}+w^{n-\theta})g^{n}, \quad  1\leq n\leq N,
\end{align*}
it holds that
\begin{align*}%\label{gronwall2}
u^n+v^n+w^n\leq  6E_{\beta}(6\max(1,\rho)\pi_A\Lambda t_n^\beta)\left(u^0+v^0+w^0+\max_{1\leq k\leq n}\sum_{j=1}^kP_{k-j}^{(k)}g^{j} \right)\quad \mbox{for}~1\leq n\leq N.
\end{align*}
\end{lemma}

Similar to \eqref{P-proper}, the discrete complementary convolution kernels $P_{n-j}^{(n)}$ in the above lemma fulfill
\begin{align*}
0\leq P_{n-j}^{(n)}\leq \pi_A\Gamma(2-\beta)\tau_j^\beta,\quad \sum_{j=1}^nP_{n-j}^{(n)}\omega_{1-\beta}(t_j)\leq \pi_A,\quad 1\leq j\leq n\leq N.
\end{align*}

\begin{theorem}\label{Stability-DW}
If {\bf V1} is valid, the numerical scheme \eqref{dw-sc1}--\eqref{dw-sc2} is unconditionally stable and the discrete solutions $u_h^n~({\bf x}_h\in \Omega_h,1\leq n\leq N)$  satisfy
\begin{align}\nonumber
\|\nabla_h u^n\|\leq& C \left[ \|\nabla_h \varphi\|+t_n\|\nabla_h\psi\|+\max_{1\leq k\leq n}\sum_{j=1}^k P_{k-j}^{(k)} (\|f^{j-\theta}\|+\|{\tilde \psi}^{j-\theta}\|)\right]\\\label{DW-stb}
 \leq& C \left[\|\nabla_h \varphi \|+t_n\|\nabla_h\psi\|+\max_{1\leq k\leq n}\{t_k^{\frac{\alpha}{2}} (\|f^{k-\theta}\|+\|{\tilde \psi}^{k-\theta}\|)\}\right].
\end{align}
\end{theorem}
\begin{proof}
Multiplying both sides of \eqref{dw-matrix-form-1} by $({\bf P}^{n-\theta}{\bf v}^{n-\theta})^T$  yields
\begin{align}\nonumber
&({\bf v}^{n-\theta})^T{\bf P}^{n-\theta}({\cal D}_{\tau}^\beta {\bf v})^{n-\theta}+({\bf v}^{n-\theta})^T{\bf A}^{n-\theta}{\tilde{\bf u}}^{n-\theta}\\\label{stb1-DW}
=&({\bf v}^{n-\theta})^T{\bf P}^{n-\theta}\left({\bf B}^{n-\theta}+{\bf C}^{n-\theta}\right){\tilde{\bf u}}^{n-\theta}+({\bf v}^{n-\theta})^T{\bf P}^{n-\theta}({\bf f}^{n-\theta}+{\bf \Psi}^{n-\theta}).
\end{align}
On the other hand, the multiplication of  $({\tilde{\bf u}}^{n-\theta})^T{\bf A}^{n-\theta}$ on both sides of \eqref{dw-matrix-form-2} gives
\begin{align}\label{stb2-DW}
({\tilde{\bf u}}^{n-\theta})^T{\bf A}^{n-\theta}{\bf v}^{n-\theta}=({\tilde{\bf u}}^{n-\theta})^T{\bf A}^{n-\theta}({\cal D}_{\tau}^\beta {\tilde{\bf u}})^{n-\theta}.
\end{align}
Thus, it follows from \eqref{stb1-DW} and \eqref{stb2-DW} that
\begin{align}\nonumber
&({\bf v}^{n-\theta})^T{\bf P}^{n-\theta}({\cal D}_{\tau}^\beta {\bf v})^{n-\theta}+({\tilde{\bf u}}^{n-\theta})^T{\bf A}^{n-\theta}({\cal D}_{\tau}^\beta {\tilde{\bf u}})^{n-\theta}\\\label{stb3-DW}
=&({\bf v}^{n-\theta})^T{\bf P}^{n-\theta}\left({\bf B}^{n-\theta}+{\bf C}^{n-\theta}\right){\tilde{\bf u}}^{n-\theta}+({\bf v}^{n-\theta})^T{\bf P}^{n-\theta}({\bf f}^{n-\theta}+{\bf \Psi}^{n-\theta}).
\end{align}
The first and second terms on the left-hand side of \eqref{stb3-DW} are evaluated by means of \eqref{uPDu} and \eqref{uADu}, respectively. Then we consider the terms on the right-hand side. Applying the Cauchy-Schwarz inequality and utilizing \eqref{stab2-1}, we have
\begin{align}\nonumber
&2({\bf v}^{n-\theta})^T{\bf P}^{n-\theta}{\bf B}^{n-\theta}{\tilde{\bf u}}^{n-\theta}\\\nonumber
\leq&({\bf v}^{n-\theta})^T({\bf P}^{n-\theta})^2{\bf v}^{n-\theta}+
({\tilde{\bf u}}^{n-\theta})^T\left[({\bf B}^{n-\theta})^T{\bf B}^{n-\theta}\right]{\tilde{\bf u}}^{n-\theta}\\\label{stb4-DW}
\leq&\frac1{{\tilde C}_0}({\bf v}^{n-\theta})^T{\bf P}^{n-\theta}{\bf v}^{n-\theta}+ C_5\left[ ({\tilde{\bf u}}_x^{n-\theta})^T{\bf P}_1^{n-\theta}{\tilde{\bf u}}_x^{n-\theta}+({\tilde{\bf u}}_y^{n-\theta})^T{\bf P}_2^{n-\theta}{\tilde{\bf u}}_y^{n-\theta} \right].
\end{align}
With the embedding inequality \eqref{embedding-ine}, one  has
\begin{align}\nonumber
&2({\bf v}^{n-\theta})^T{\bf P}^{n-\theta}{\bf C}^{n-\theta}{\tilde{\bf u}}^{n-\theta}\\\nonumber
\leq&({\bf v}^{n-\theta})^T({\bf P}^{n-\theta})^2{\bf v}^{n-\theta}+
({\tilde{\bf u}}^{n-\theta})^T({\bf C}^{n-\theta})^2 {\tilde{\bf u}}^{n-\theta}\\\nonumber
\leq&\frac1{{\tilde C}_0}({\bf v}^{n-\theta})^T{\bf P}^{n-\theta}{\bf v}^{n-\theta}+ C_3^2C_\Omega^2\left[ ({\tilde{\bf u}}_x^{n-\theta})^T{\tilde{\bf u}}_x^{n-\theta}+({\tilde{\bf u}}_y^{n-\theta})^T{\tilde{\bf u}}_y^{n-\theta} \right]\\\label{stb5-DW}
\leq&\frac1{{\tilde C}_0}({\bf v}^{n-\theta})^T{\bf P}^{n-\theta}{\bf v}^{n-\theta}+ C_6\left[ ({\tilde{\bf u}}_x^{n-\theta})^T{\bf P}_1^{n-\theta}{\tilde{\bf u}}_x^{n-\theta}+({\tilde{\bf u}}_y^{n-\theta})^T{\bf P}_2^{n-\theta}{\tilde{\bf u}}_y^{n-\theta} \right].
\end{align}
Hence, from \eqref{stb3-DW}--\eqref{stb5-DW}, \eqref{uPDu} and \eqref{uADu}, we obtain
\begin{align}\nonumber
&\sum_{k=1}^nA_{n-k}^{(n)}\nabla_\tau \left[({\bf v}^k)^T{\bf P}^{k-\theta}{\bf v}^k +({\tilde{\bf u}}_x^k)^T{\bf P}_1^{k-\theta}{\tilde{\bf u}}_x^k+ ({\tilde{\bf u}}_y^k)^T{\bf P}_2^{k-\theta}{\tilde{\bf u}}_y^k\right]\\\nonumber
\leq&2\max\{\frac1{{\tilde C}_0}, C_5,C_6\}\left[({\bf v}^{n-\theta})^T{\bf P}^{n-\theta}{\bf v}^{n-\theta}+ ({\tilde{\bf u}}_x^{n-\theta})^T{\bf P}_1^{n-\theta}{\tilde{\bf u}}_x^{n-\theta}+({\tilde{\bf u}}_y^{n-\theta})^T{\bf P}_2^{n-\theta}{\tilde{\bf u}}_y^{n-\theta} \right]\\\label{stb6-DW}
&+ 2\left[({\bf P}^{n-\theta})^{\frac12}{\bf v}^{n-\theta}\right]^T({\bf P}^{n-\theta})^{\frac12} ({\bf f}^{n-\theta}+{\bf \Psi}^{n-\theta}).
\end{align}
Multiplying both sides of \eqref{stb6-DW} by $h_xh_y$ and taking $C_8:=2\max\{\frac1{{\tilde C}_0}, C_5,C_6\}$, we further get
\begin{align}\nonumber
&\sum_{k=1}^nA_{n-k}^{(n)}\nabla_\tau \left[\|v^k\|_P^2 +\|\delta_x {\tilde u}^k\|_{P_1}^2+ \|\delta_y {\tilde u}^k\|_{P_2}^2 \right]\\\nonumber
\leq& C_8\left(\|v^{n-\theta}\|_P^2+ \|\delta_x {\tilde u}^{n-\theta}\|_{P_1}^2+\|\delta_y {\tilde u}^{n-\theta}\|_{P_2}^2\right)
+\sqrt{\hat C_0}\|v^{n-\theta}\|_P\cdot \|f^{n-\theta}+{\tilde \psi}^{n-\theta}\|\\\nonumber
\leq&C_8\left[ \left(\|v\|_P^{(n-\theta)}\right)^2+ \left(\|\delta_x {\tilde u}\|_{P_1}^{(n-\theta)}\right)^2+\left(\|\delta_y {\tilde u}\|_{P_2}^{(n-\theta)}\right)^2\right]\\\label{stb7-DW}
&+\sqrt{\hat C_0} \left(\|v\|_P^{(n-\theta)}+\|\delta_x {\tilde u}\|_{P_1}^{(n-\theta)}+\|\delta_y {\tilde u} \|_{P_2}^{(n-\theta)}\right) (\|f^{n-\theta}\|+\|{\tilde \psi}^{n-\theta}\|),
\end{align}
where $\|v^k\|_P^2:=h_xh_y({\bf v}^k)^T{\bf P}^{k-\theta}{\bf v}^k$.%, which fulfills $\frac1{{\hat C}_0}\|v^k\|^2\leq \|v^k\|_P^2\leq \frac1{{\tilde C}_0}\|v^k\|^2$.

Now, combining \eqref{stb7-DW} with the fractional Gr\"{o}nwall inequality (Lemma \ref{gronwall-lemma2}), it follows
\begin{align*}
\|v^n\|_{P} +\|\delta_x {\tilde u}^n\|_{P_1} + \|\delta_y {\tilde u}^n\|_{P_2}
\leq& 6E_{\beta}(6\max(1,\rho)\pi_AC_8 t_n^\beta) \Big[\|v^0\|_{P}+\|\delta_x {\tilde u}^0\|_{P_1} + \|\delta_y {\tilde u}^0\|_{P_2}\\
&+\sqrt{\hat C_0}\max_{1\leq k\leq n}\sum_{j=1}^k P_{k-j}^{(k)} (\|f^{j-\theta}\|+\|{\tilde \psi}^{j-\theta}\|)\Big],
\end{align*}
and hence
\begin{align*}
\|\nabla_h {\tilde u}^n\|\leq \|\delta_x {\tilde u}^n\|+\|\delta_y {\tilde u}^n\|
\leq& \max\{\frac1{\sqrt{{\tilde C}_1}},\frac1{\sqrt{{\tilde C}_2}}\}\cdot\left( \|\delta_x {\tilde u}^n\|_{P_1} + \|\delta_y {\tilde u}^n\|_{P_2}\right)\\
\leq& C \left[\|v^0\|_{P}+ \|\nabla_h {\tilde u}^0\|+\max_{1\leq k\leq n}\sum_{j=1}^k P_{k-j}^{(k)} (\|f^{j-\theta}\|+\|{\tilde \psi}^{j-\theta}\|)\right]\\
\leq&C \left[\|v^0\|_{P}+\|\nabla_h {\tilde u}^0\|+\max_{1\leq k\leq n}\{t_k^\beta (\|f^{k-\theta}\|+\|{\tilde \psi}^{k-\theta}\|)\}\right].
\end{align*}
Then the claimed result \eqref{DW-stb} can be reached by the properties $\|\nabla_h u^n\|\leq \|\nabla_h {\tilde u}^n\|+t_n\|\nabla_h\psi\|$, $\|v^0\|_{P}=0$ and $ \|\nabla_h {\tilde u}^0\|=\|\nabla_h \varphi\|$.

\end{proof}

The next theorem shows the convergence of proposed scheme \eqref{dw-matrix-form-1}--\eqref{dw-matrix-form-2}.
\begin{theorem}\label{Convergence-dw}
Denote $e_h^k:=u({\bf x}_h,t_k)-u_h^k~({{\bf x}_h}\in\bar\Omega_h,~0\leq k\leq N)$ . If {\bf V1}, {\bf V2}, {\bf MA} and the assumptions in \eqref{regularity1} and \eqref{regularity3} are valid, the numerical scheme \eqref{dw-sc1}--\eqref{dw-sc2} is unconditionally convergent with
\begin{equation}
\|\nabla_h e^n\|\leq
C(\tau^{\min\{2,\gamma\sigma_2,\gamma\sigma_3\}}+h_x^2+h_y^2),
\quad \mbox{for}\quad 1\leq n\leq N.
\end{equation}
\end{theorem}

\begin{proof}
We have
$$e_h^k=u({\bf x}_h,t_k)-u_h^k={\tilde u}({\bf x}_h,t_k)-{\tilde u}_h^k,\quad {\bf x}_h\in\Omega_h,~ 1\leq k\leq N.$$
Denote ${\check e}_h^k:=v({\bf x}_h,t_k)-v_h^k$ ($1\leq k\leq N$) and
\begin{align}\label{TE-dw}
{\tilde R}_h^{n-\theta}:=-{\cal T}_{v1}({\bf x}_h,t_{n-\theta})+{\cal T}_A({\bf x}_h,t_{n-\theta})+{\cal S}({\bf x}_h,t_{n-\theta}),~ {\hat R}_h^{n-\theta}:=-{\cal T}_{v2}({\bf x}_h,t_{n-\theta})+{\cal T}_{\tilde u}({\bf x}_h,t_{n-\theta}),
\end{align}
for ${\bf x}_h\in\Omega_h$, where the above truncation errors are discussed in the Appendix (Subsection \ref{TEA}).

Denote the vector
$${\check{\bf e}}^k:=({\check e}_{1,1}^k,\cdots,{\check e}_{M_x-1,1}^k,\cdots,{\check e}_{1,2}^k,\cdots,{\check e}_{M_x-1,2}^k,\cdots\cdots,{\check e}_{1,M_y-1}^k,\cdots,{\check e}_{M_x-1,M_y-1}^k),$$
 while ${\tilde {\bf R}}^{n-\theta}$, ${\hat {\bf R}}^{n-\theta}$ are similarity  defined  with entries ${\tilde R}_h^{n-\theta}$ and ${\hat R}_h^{n-\theta}$, respectively.

The error equations to scheme \eqref{dw-matrix-form-1}--\eqref{dw-matrix-form-2}  can be given as
\begin{align}\label{dw-error1}
&({\cal D}_{\tau}^\beta {\check{\bf e}})^{n-\theta}=\left[-({\bf P}^{n-\theta})^{-1}{\bf A}^{n-\theta}+{\bf B}^{n-\theta}+{\bf C}^{n-\theta}\right]{\bf e}^{n-\theta}+{\tilde {\bf R}}^{n-\theta};\\\label{dw-error2}
& {\check{\bf e}}^{n-\theta}=({\cal D}_{\tau}^\beta {\bf e})^{n-\theta}+{\hat {\bf R}}^{n-\theta}.
\end{align}
The proof  of  convergence  is similar to that of Theorem \ref{Stability-DW} with a slight difference only at the step for \eqref{stb2-DW}.
We now have:
\begin{align*}%\label{stb2-DW}
({\bf e}^{n-\theta})^T{\bf A}^{n-\theta}{\check{\bf e}}^{n-\theta}=({\bf e}^{n-\theta})^T{\bf A}^{n-\theta}({\cal D}_{\tau}^\beta {\bf e})^{n-\theta}+({\bf e}^{n-\theta})^T{\bf A}^{n-\theta}{\hat {\bf R}}^{n-\theta}.
\end{align*}
The term can be estimated like that in \eqref{stb4}. Going though the remaining part of the proof like that of Theorem \ref{Stability-DW}, one should find no difficulty to obtain
\begin{align*}%\label{DW-stb}
\|{\check{ e}}^n\|+\|\nabla_h e^n\|\leq& C \max_{1\leq k\leq n}\sum_{j=1}^k P_{k-j}^{(k)} (\|{\tilde R}^{j-\theta}\|+\|\nabla_h {\hat R}^{j-\theta}\|).
\end{align*}
Thus, combining with \eqref{S}, \eqref{TA} and \eqref{nab_Ttildeu}--\eqref{nab_Tv2}, the desired result is true.
\end{proof}

\section{Numerical Experiments}\label{Numerical}
Numerical examples will be provided in this section to show the accuracy and efficiency of proposed schemes \eqref{sub-sc1}--\eqref{sub-sc2} and \eqref{dw-sc1}--\eqref{dw-sc2}. The variable coefficients in the two examples in this section are chose as
\begin{align*}
&a_1({\bf x},t)=e^{x+y}(1+\cos(t)),\quad a_2({\bf x},t)=e^{(x+y)t}(1+t^{\frac32}),\\
&b_1({\bf x},t)=\sin(xyt),\quad b_2({\bf x},t)=\cos(xyt),\quad b_3({\bf x},t)=(x^2+y^2)t.
\end{align*}
The above variable coefficients satisfy {\bf V1} and {\bf V2} clearly. Then we take the  function $d({\bf x})$ and constant $C_p$ in Lemma \ref{functionP} as  follows
$$d({\bf x})=e^{\sin(x+y)}\quad \mbox{and} \quad  C_p=3.$$
Since the problem we considered in the paper are linear fractional evolution equations, we choose the classical graded mesh  $t_k=T(k/N)^\gamma$ for the time partition to compensate for the lack of smoothness of the solution near the initial time. The graded mesh is definitely  in accordance with the mesh assumption {\bf MA}. In all of the numerical tests, we take $M=M_x=M_y$, the discrete $H^1$-norm errors $E_1(M,N)=\max_{1\leq n\leq N}\|U^n-u^n\|_{H^1}$ will be recorded in each run, and the temporal and spatial convergence orders are given by
         $$\mbox{Order}_\tau=\log_2\left[\frac{E_1(M,N/2)}{E_1(M,N)}\right] \quad \mbox{and}\quad
         \mbox{Order}_h=\log_2\left[\frac{E_1(M/2,N)}{E_1(M,N)}\right],$$
respectively.

 Moreover, we will always employ the sum-of-exponentials (SOE) technique \cite{FL1} to the proposed schemes while discretizing the Caputo derivative to save the memory and computational costs, since the SOE method does not bring any additional essential differences to the numerical analysis of the nonuniform schemes. One may refer to \cite[Section 5.1]{Liao-AC2} for the details of the fast Alikhanov formula and refer to \cite{FL1,LyuVongANM2020} for the advantage of the SOE approximation in the computational aspect.
 The absolute tolerance error $\epsilon$ and the cut-off time $\Delta t$ of the fast Alikhanov formula (see \cite[Lemma 5.1]{Liao-AC2}) are set as $\epsilon=10^{-12}$ and $\Delta t=\tau_1$ in all of the following tests.

\begin{example}\label{ex1}
We consider the sub-diffusion problem \eqref{eq1}--\eqref{eq2} with $\Omega=(0,1)^2$, $T=1$, $\varphi=\sin(\pi x)\sin(\pi y)$ and
$$f(u,{\bf x},t)=\sin(\pi x)\sin(\pi y)\left[\Gamma(\alpha+1)+\frac{t^{1-\alpha}}{\Gamma(2-\alpha)}\right]-{\cal A}(\sin(\pi x)\sin(\pi y))(1+t+t^\alpha),\quad \alpha\in(0,1),$$
 such that the exact solution is $u=\sin(\pi x)\sin(\pi y)(1+t+t^\alpha)$.
\end{example}

One may notice that the regularity parameter in \eqref{regularity1} should be $\sigma_1=\alpha$ for Example \ref{ex1}. Therefore, according to Theorem \ref{Convergence-sub}, the optimal mesh parameter is $\gamma_{opt}=2/\alpha$ for the scheme \eqref{sub-sc1}--\eqref{sub-sc2} on the graded time meshes.

The temporal accuracy by applying the scheme \eqref{sub-sc1}--\eqref{sub-sc2} with fixed $M=1000$ and different parameters $\alpha,\gamma$ for solving Example \ref{ex1} is listed in Tables \ref{table1}--\ref{table3}, while Table \ref{table4} shows the spatial accuracy with fixed $N=500$. From the four tables, we can clearly observe the optimal second-order accuracy of the proposed scheme, and the optimal choice of the grading parameter ($\gamma_{opt}=2/\alpha$) is well reflected.

\begin{table}[htb!]
 \begin{center}
 \caption{Numerical accuracy  in temporal direction of the scheme \eqref{sub-sc1}--\eqref{sub-sc2} for solving Example \ref{ex1}, where $\alpha=0.5$.}\label{table1}
 \renewcommand{\arraystretch}{1}
 \def\temptablewidth{0.95\textwidth}
 {\rule{\temptablewidth}{1pt}}
 \begin{tabular*}{\temptablewidth}{@{\extracolsep{\fill}}ccccccc}
&\multicolumn{2}{c}{$\gamma=1$}  &\multicolumn{2}{c}{$\gamma_{opt}=2/\alpha=4$}  &\multicolumn{2}{c}{$\gamma=2.5/\alpha=5$}\\
  \cline{2-3}\cline{4-5}\cline{6-7}
 $N$ &  $E_1(M,N)$ &Order$_\tau$ & $E_1(M,N)$ &Order$_\tau$ & $E_1(M,N)$ &Order$_\tau$\\\hline
 $4$  & 1.6124e-01  & $\ast$   & 4.4642e-02  & $\ast$  & 6.4526e-02  & $\ast$ \\
 $8$  & 1.1090e-01  & 0.54  & 1.2036e-02  & 1.89  & 1.8154e-02  & 1.83 \\
$16$  & 7.5477e-02  & 0.56  & 3.1256e-03  & 1.95  &4.7969e-03 & 1.92 \\
$32$ & 5.0612e-02  & 0.58   &8.0038e-04  & 1.97  & 1.2989e-03 & 1.88\\\hline
Theoretical Order && 0.50 && 2.00 && 2.00
\end{tabular*}
{\rule{\temptablewidth}{1pt}}
\end{center}
\end{table}

\begin{table}[htb!]
 \begin{center}
 \caption{Numerical accuracy  in temporal direction of the scheme \eqref{sub-sc1}--\eqref{sub-sc2} for solving Example \ref{ex1}, where $\alpha=0.7$.}\label{table2}
 \renewcommand{\arraystretch}{1}
 \def\temptablewidth{0.95\textwidth}
 {\rule{\temptablewidth}{1pt}}
 \begin{tabular*}{\temptablewidth}{@{\extracolsep{\fill}}ccccccc}
&\multicolumn{2}{c}{$\gamma=1$}  &\multicolumn{2}{c}{$\gamma_{opt}=2/\alpha\approx 2.86$}  &\multicolumn{2}{c}{$\gamma=2.5/\alpha\approx3.57$}\\
  \cline{2-3}\cline{4-5}\cline{6-7}
 $N$ &  $E_1(M,N)$ &Order$_\tau$ & $E_1(M,N)$ &Order$_\tau$ & $E_1(M,N)$ &Order$_\tau$\\\hline
 $4$  & 1.0592e-01  & $\ast$   & 2.5978e-02  & $\ast$  & 3.8595e-02  & $\ast$ \\
 $8$  & 6.1506e-02  & 0.78  & 6.5510e-03  & 1.99  & 1.0043e-02  & 1.94 \\
$16$  & 3.4534e-02  & 0.83  & 1.6656e-03  & 1.98  &2.5747e-03 & 1.96 \\
$32$ &1.8403e-02  & 0.91   &4.2368e-04  & 1.98  & 6.5520e-04 & 1.97\\\hline
Theoretical Order && 0.70 && 2.00 && 2.00
\end{tabular*}
{\rule{\temptablewidth}{1pt}}
\end{center}
\end{table}

\begin{table}[htb!]
 \begin{center}
 \caption{Numerical accuracy  in temporal direction of the scheme \eqref{sub-sc1}--\eqref{sub-sc2} for solving Example \ref{ex1}, where $\alpha=0.9$.}\label{table3}
 \renewcommand{\arraystretch}{1}
 \def\temptablewidth{0.95\textwidth}
 {\rule{\temptablewidth}{1pt}}
 \begin{tabular*}{\temptablewidth}{@{\extracolsep{\fill}}ccccccc}
&\multicolumn{2}{c}{$\gamma=1$}  &\multicolumn{2}{c}{$\gamma_{opt}=2/\alpha\approx 2.22$}  &\multicolumn{2}{c}{$\gamma=2.5/\alpha\approx2.78$}\\
  \cline{2-3}\cline{4-5}\cline{6-7}
 $N$ &  $E_1(M,N)$ &Order$_\tau$ & $E_1(M,N)$ &Order$_\tau$ & $E_1(M,N)$ &Order$_\tau$\\\hline
 $4$  & 3.4213e-02  & $\ast$   & 8.0750e-03  & $\ast$  & 1.2235e-02  & $\ast$ \\
 $8$  & 1.6299e-02  & 1.07   & 1.8998e-03  & 2.09  & 2.9302e-03  & 2.06 \\
$16$  & 7.0935e-03  & 1.20  & 4.7949e-04  & 1.99  &7.4266e-04 & 1.98 \\
$32$ &2.6405e-03  & 1.43   & 1.2448e-04 & 1.95  & 1.9062e-04 & 1.96\\\hline
Theoretical Order && 0.90 && 2.00 && 2.00
\end{tabular*}
{\rule{\temptablewidth}{1pt}}
\end{center}
\end{table}

\begin{table}[htb!]
 \begin{center}
 \caption{Numerical accuracy  in spatial direction of the scheme \eqref{sub-sc1}--\eqref{sub-sc2} for solving Example \ref{ex1}, where $\alpha=0.7$.}\label{table4}
 \renewcommand{\arraystretch}{1}
 \def\temptablewidth{0.95\textwidth}
 {\rule{\temptablewidth}{1pt}}
 \begin{tabular*}{\temptablewidth}{@{\extracolsep{\fill}}ccccccc}
&\multicolumn{2}{c}{$\gamma=1$}  &\multicolumn{2}{c}{$\gamma_{opt}=2/\alpha\approx 2.86$}  &\multicolumn{2}{c}{$\gamma=2.5/\alpha\approx3.57$}\\
  \cline{2-3}\cline{4-5}\cline{6-7}
 $M$ &  $E_1(M,N)$ &Order$_h$ & $E_1(M,N)$ &Order$_h$ & $E_1(M,N)$ &Order$_h$\\\hline
 $4$  & 3.6943e-01  & $\ast$   & 3.6942e-01 & $\ast$  & 3.6931e-01  & $\ast$ \\
 $8$  & 9.1710e-02  & 2.01   & 9.1710e-02  & 2.01  & 9.1666e-02  & 2.01 \\
$16$  & 2.2891e-02  & 2.00  & 2.2891e-02  & 2.00  &2.2864e-02 & 2.00 \\
$32$ & 5.7205e-03  & 2.00   & 5.7213e-03 & 2.00  & 5.6977e-03 & 2.00\\\hline
%$64$ & 1.4302e-03  & 2.00   & 1.4309e-03 & 2.00  & 1.4084e-03 & 2.02\\\hline
Theoretical Order && 2.00 && 2.00 && 2.00
\end{tabular*}
{\rule{\temptablewidth}{1pt}}
\end{center}
\end{table}

\begin{example}\label{ex2}
We then consider the diffusion-wave problem \eqref{eq1} and \eqref{eq3} with $\Omega=(0,1)^2$, $T=1$, $\phi=\psi=\sin(\pi x)\sin(\pi y)$ and
$$f(u,{\bf x},t)=\Gamma(\alpha+1)\sin(\pi x)\sin(\pi y)-{\cal A}(\sin(\pi x)\sin(\pi y))(1+t+t^\alpha),\quad \alpha\in(1,2),$$
such that the exact solution is $u=\sin(\pi x)\sin(\pi y)(1+t+t^\alpha)$.
\end{example}

For Example \ref{ex2}, the regularity parameters in \eqref{regularity2} are $\sigma_2=\alpha$ and $\sigma_3=\alpha/2$. Then, Theorem \ref{Convergence-dw} indicates that the optimal mesh parameter is $\gamma_{opt}=2/\sigma_3=4/\alpha$ for the scheme \eqref{dw-sc1}--\eqref{dw-sc2} on the graded time meshes. One can notice that the grading parameter $\gamma_{opt}$ is bounded and not large while $\alpha\to1^{+}$, this keeps the robustness of the graded scheme in practical applications when the fractional order $\alpha$ is close to one.

Similarly, we display the temporal accuracy, which is obtained by applying the scheme \eqref{dw-sc1}--\eqref{dw-sc2} with fixed $M=1000$ and different parameters for solving Example \ref{ex2}, in Tables \ref{table5}--\ref{table8}.
The spatial accuracy of the scheme with fixed $N=500$ is displayed in Table \ref{table9}. The numerical results show that the proposed scheme \eqref{dw-sc1}--\eqref{dw-sc2} also works very well with optimal second-order accuracy and is robust for $\alpha\to1^{+}$ in solving the diffusion-wave problem with general variable coefficients.

\begin{table}[htb!]
 \begin{center}
 \caption{Numerical accuracy  in temporal direction of scheme \eqref{dw-sc1}--\eqref{dw-sc2} for Example \ref{ex2}, where $\alpha=1.01$.}\label{table5}
 \renewcommand{\arraystretch}{1}
 \def\temptablewidth{0.95\textwidth}
 {\rule{\temptablewidth}{1pt}}
 \begin{tabular*}{\temptablewidth}{@{\extracolsep{\fill}}ccccccc}
&\multicolumn{2}{c}{$\gamma=1$}  &\multicolumn{2}{c}{$\gamma_{opt}=4/\alpha\approx 3.96$}  &\multicolumn{2}{c}{$\gamma=4.5/\alpha\approx4.46$}\\
  \cline{2-3}\cline{4-5}\cline{6-7}
 $N$ &  $E_1(M,N)$ &Order$_\tau$ & $E_1(M,N)$ &Order$_\tau$ & $E_1(M,N)$ &Order$_\tau$\\\hline
 $4$  & 1.2885e-02  & $\ast$   & 4.7702e-03  & $\ast$  & 4.7959e-03  & $\ast$ \\
 $8$  & 1.1231e-02  & 0.66   & 1.5632e-03  & 1.85  & 1.4205e-03  & 1.76 \\
 %$8$  &  1.1231e-02 & $\ast$   & 1.5632e-03  & $\ast$  & 1.4205e-03  & $\ast$ \\
$16$  & 9.2424e-03  & 0.28  & 4.2372e-04  & 1.88  &4.0423e-04 & 1.81 \\
$32$ & 6.9173e-03  & 0.42   & 1.0616e-04 & 2.00  & 1.0064e-04 & 2.01\\\hline
%$64$ & 4.6212e-03  & 0.58   & 1.7268e-05 & 2.62  & 2.3505e-04 & -1.22\\\hline
Theoretical Order && 0.505 && 2.00 && 2.00
\end{tabular*}
{\rule{\temptablewidth}{1pt}}
\end{center}
\end{table}

\begin{table}[htb!]
 \begin{center}
 \caption{Numerical accuracy  in temporal direction of scheme \eqref{dw-sc1}--\eqref{dw-sc2} for Example \ref{ex2}, where $\alpha=1.1$.}\label{table6}
 \renewcommand{\arraystretch}{1}
 \def\temptablewidth{0.95\textwidth}
 {\rule{\temptablewidth}{1pt}}
 \begin{tabular*}{\temptablewidth}{@{\extracolsep{\fill}}ccccccc}
&\multicolumn{2}{c}{$\gamma=1$}  &\multicolumn{2}{c}{$\gamma_{opt}=4/\alpha\approx 3.64$}  &\multicolumn{2}{c}{$\gamma=4.5/\alpha\approx4.09$}\\
  \cline{2-3}\cline{4-5}\cline{6-7}
 $N$ &  $E_1(M,N)$ &Order$_\tau$ & $E_1(M,N)$ &Order$_\tau$ & $E_1(M,N)$ &Order$_\tau$\\\hline
 $4$  & 2.4901e-02  & $\ast$   & 1.6750e-02  & $\ast$  & 2.0056e-02  & $\ast$ \\
 $8$  & 1.5761e-02  & 0.66   & 4.6593e-03  & 1.85  & 5.7785e-03  & 1.80 \\
 %$8$  & 1.5761e-02  & $\ast$   & 4.6593e-03  & $\ast$  & 5.7785e-03  & $\ast$ \\
$16$  & 1.0245e-02  & 0.62  & 1.2195e-03  & 1.93  &1.5306e-03 & 1.92 \\
$32$ & 6.3732e-03  & 0.68   & 3.0860e-04 & 1.98  & 3.9009e-04 & 1.97\\\hline
%$64$ & 3.6444e-03  & 0.81   & 7.6077e-05 & 2.02  & 8.9817e-05 & 2.12\\\hline
Theoretical Order && 0.55 && 2.00 && 2.00
\end{tabular*}
{\rule{\temptablewidth}{1pt}}
\end{center}
\end{table}

\begin{table}[htb!]
 \begin{center}
 \caption{Numerical accuracy  in temporal direction of scheme \eqref{dw-sc1}--\eqref{dw-sc2} for Example \ref{ex2}, where $\alpha=1.5$.}\label{table7}
 \renewcommand{\arraystretch}{1}
 \def\temptablewidth{0.95\textwidth}
 {\rule{\temptablewidth}{1pt}}
 \begin{tabular*}{\temptablewidth}{@{\extracolsep{\fill}}ccccccc}
&\multicolumn{2}{c}{$\gamma=1$}  &\multicolumn{2}{c}{$\gamma_{opt}=4/\alpha\approx 2.67$}  &\multicolumn{2}{c}{$\gamma=4.5/\alpha=3$}\\
  \cline{2-3}\cline{4-5}\cline{6-7}
 $N$ &  $E_1(M,N)$ &Order$_\tau$ & $E_1(M,N)$ &Order$_\tau$ & $E_1(M,N)$ &Order$_\tau$\\\hline
 $4$  &5.3444e-02  & $\ast$   & 7.8413e-02  & $\ast$  & 9.6727e-02  & $\ast$ \\
 $8$  & 1.7243e-02  & 1.63   & 2.0766e-02  & 1.92  & 2.5910e-02  & 1.90 \\
 % $8$  & 1.7243e-02  & $\ast$   & 2.0766e-02  & $\ast$  & 2.5910e-02  & $\ast$ \\
$16$  & 6.1521e-03  & 1.49  & 5.3057e-03  & 1.97  &6.6772e-03 & 1.96 \\
$32$ & 2.3596e-03  & 1.38   & 1.3373e-03 & 1.99  & 1.6881e-03 & 1.98\\\hline
%$64$ & 1.0008e-03  & 1.24   & 3.3305e-04 & 2.01  & 4.2172e-04 & 2.00\\\hline
Theoretical Order && 0.75 && 2.00 && 2.00
\end{tabular*}
{\rule{\temptablewidth}{1pt}}
\end{center}
\end{table}

\begin{table}[htb!]
 \begin{center}
 \caption{Numerical accuracy  in temporal direction of scheme \eqref{dw-sc1}--\eqref{dw-sc2} for Example \ref{ex2}, where $\alpha=1.9$.}\label{table8}
 \renewcommand{\arraystretch}{1}
 \def\temptablewidth{0.95\textwidth}
 {\rule{\temptablewidth}{1pt}}
 \begin{tabular*}{\temptablewidth}{@{\extracolsep{\fill}}ccccccc}
&\multicolumn{2}{c}{$\gamma=1$}  &\multicolumn{2}{c}{$\gamma_{opt}=4/\alpha\approx 2.11$}  &\multicolumn{2}{c}{$\gamma=4.5/\alpha\approx2.37$}\\
  \cline{2-3}\cline{4-5}\cline{6-7}
 $N$ &  $E_1(M,N)$ &Order$_\tau$ & $E_1(M,N)$ &Order$_\tau$ & $E_1(M,N)$ &Order$_\tau$\\\hline
 $4$  & 6.1480e-02  & $\ast$   & 1.3149e-01  & $\ast$  & 1.6643e-01  & $\ast$ \\
 $8$  & 1.6140e-02  & 1.93   & 3.0479e-02  & 2.11  & 3.8947e-02  & 2.10 \\
 % $8$  & 1.6140e-02  & $\ast$   & 3.0479e-02  & $\ast$  & 3.8947e-02  & $\ast$ \\
$16$  & 4.1813e-03  & 1.95  & 7.8883e-03  & 1.95  &9.9179e-03 & 1.97 \\
$32$ & 1.1127e-03  & 1.91   & 2.0132e-03 & 1.97  & 2.5274e-03 & 1.97\\\hline
%$64$ & 3.2595e-04  & 1.77   &5.0336e-04  & 2.00  & 6.3296e-04 & 2.00\\\hline
Theoretical Order && 0.95 && 2.00 && 2.00
\end{tabular*}
{\rule{\temptablewidth}{1pt}}
\end{center}
\end{table}

\begin{table}[htb!]
 \begin{center}
 \caption{Numerical accuracy  in spatial direction of scheme \eqref{dw-sc1}--\eqref{dw-sc2} for Example \ref{ex2}, where $\alpha=1.5$.}\label{table9}
 \renewcommand{\arraystretch}{1}
 \def\temptablewidth{0.95\textwidth}
 {\rule{\temptablewidth}{1pt}}
 \begin{tabular*}{\temptablewidth}{@{\extracolsep{\fill}}ccccccc}
&\multicolumn{2}{c}{$\gamma=1$}  &\multicolumn{2}{c}{$\gamma_{opt}=4/\alpha\approx 2.67$}  &\multicolumn{2}{c}{$\gamma=4.5/\alpha=3$}\\
  \cline{2-3}\cline{4-5}\cline{6-7}
 $M$ &  $E_1(M,N)$ &Order$_h$ & $E_1(M,N)$ &Order$_h$ & $E_1(M,N)$ &Order$_h$\\\hline
 $4$  & 2.4719e-01  & $\ast$   & 2.4718e-01 & $\ast$  & 2.4718e-01  & $\ast$ \\
 $8$  & 6.1357e-02  & 2.01   & 6.1352e-02  & 2.01  & 6.1349e-02  & 2.01 \\
  %$8$  & 6.1357e-02  & $\ast$   & 6.1352e-02  & $\ast$  & 6.1349e-02  & $\ast$ \\
$16$  & 1.5313e-02  & 2.00  & 1.5309e-02  & 2.00  &1.5306e-02 & 2.00 \\
$32$ & 3.8262e-03  & 2.00   & 3.8214e-03 & 2.00  & 3.8190e-03 & 2.00\\\hline
%$64$ & 9.5590e-04  & 2.00   & 9.5110e-04 & 2.01  & 9.4866e-04 & 2.01\\\hline
Theoretical Order && 2.00  && 2.00 && 2.00
\end{tabular*}
{\rule{\temptablewidth}{1pt}}
\end{center}
\end{table}

\section{Conclusion}\label{Conclusion}
We introduced a novel and concise technique to study numerical methods on nonuniform time partitions  for solving  time fractional evolution equations (including the sub-diffusion and diffusion-wave equations) with general time-space dependent variable coefficients. The proposed numerical schemes utilized the Alikhanov formula on nonuniform meshes. Under reasonable assumptions on the solution regularity, the variable coefficients, and weak mesh restrictions, we showed that  the nonuniform schemes are  unconditionally stable and second-order convergent with respect to discrete $H^1$-norm. The efficiency and accuracy of proposed schemes are well verified by some numerical experiments.

\section{Appendix}\label{Appendix}

\subsection{The coefficients of Alikhanov formulas}\label{L1Alikhanov}

The coefficients $A_{n-k}^{(n)}$ of the Alikhanov formula on general meshes are defined as (\cite{LiaoSecondOrder})
\begin{align*}%\label{L2coe}
A_{n-k}^{(n)}:=\left\{\begin{array}{ll}
a_0^{(n)}+\rho_{n-1}b_1^{(n)},\quad & k=n,\\
a_{n-k}^{(n)}+\rho_{k-1}b_{n-k+1}^{(n)}-b_{n-k}^{(n)}, & 2\leq k\leq n-1,\\
a_{n-1}^{(n)}-b_{n-1}^{(n)},  & k=1,
\end{array}\right.\quad \mbox{for}~n\geq 2,
\end{align*}
where
\begin{align*}
&a_{n-k}^{(n)}:=\frac1{\tau_k}\int_{t_{k-1}}^{\min\{t_k,t_{n-\theta}\}}\omega_{1-\beta}(t_{n-\theta}-s)\zd s ,~1\leq k\leq n,\\
&b_{n-k}^{(n)}:=\frac2{\tau_k(\tau_k+\tau_{k+1})}\int_{t_{k-1}}^{t_k}\omega_{1-\beta}(t_{n-\theta}-s)(s-t_{k-\frac12})\zd s ,~1\leq k\leq n-1,
\end{align*}
with $\rho_k:=\tau_k/\tau_{k+1}$ being the local time step-size ratios.
It has been proved in \cite{LiaoGronwall,LiaoSecondOrder} that the discrete coefficients of the nonuniform Alikhanov formula (with $\pi_A=11/4$ and $\rho=7/4$, where $\rho:=\max_{k}\{\rho_k\}$ is the maximum step-size ratio) satisfy two basic properties:
\begin{itemize}
\item[{\bf A1.}] The discrete kernels are positive and monotone: $A_0^{(n)}\geq A_1^{(n)}\geq \cdots\geq A_{n-1}^{(n)}>0$;
\item[{\bf A2.}] The discrete kernels fulfill $A_{n-k}^{(n)}\geq \frac1{\pi_A}\int_{t_{k-1}}^{t_k}\frac{\omega_{1-\beta}(t_n-s)}{\tau_k}\zd s$ for $1\leq k\leq n\leq N$.
\end{itemize}
%With {\bf A1}--{\bf A2}, an important inequality in the numerical analysis will holds for the nonuniform L1 formula  \cite[proof of Theorem 2.1]{LiaoL1} and the nonuniform Alikhanov formula \cite[Corollary 2.3]{LiaoSecondOrder}:
%\begin{align}\label{product-proper1}
%\left\langle ({\cal D}_\tau^\delta g)^{n-\theta}, g^{n-\theta} \right\rangle \geq \frac12 \sum_{k=1}^nA_{n-k}^{(n)}\nabla_{\tau}(\|g^k\|^2) \quad \mbox{for}~1\leq n\leq N.
%\end{align}

\subsection{The proof of \eqref{zQDz1}}\label{zQDz0}
We will go through the proof of \cite[Lemma A.1]{LiaoGronwall} to show that
\begin{align}\label{zQDz11}
&2({\bf z}^{n})^T{\bf Q}^{(n)}({\cal D}_\tau^\beta {\bf z})^{n-\theta}\geq \sum_{k=1}^nA_{n-k}^{(n)}\nabla_\tau [({\bf z}^k)^T{\bf Q}^{(n)}{\bf z}^k]+\frac{(({\cal D}_\tau^\beta {\bf z})^{n-\theta})^T{\bf Q}^{(n)}({\cal D}_\tau^\beta {\bf z})^{n-\theta}}{A_0^{(n)}},\\\label{zQDz12}
&2({\bf z}^{n-1})^T{\bf Q}^{(n)}({\cal D}_\tau^\beta {\bf z})^{n-\theta}\geq \sum_{k=1}^nA_{n-k}^{(n)}\nabla_\tau [({\bf z}^k)^T{\bf Q}^{(n)}{\bf z}^k]-\frac{(({\cal D}_\tau^\beta {\bf z})^{n-\theta})^T{\bf Q}^{(n)}({\cal D}_\tau^\beta {\bf z})^{n-\theta}}{A_0^{(n)}-A_1^{(n)}},
\end{align}
for $1\leq n\leq N$ and $A_1^{(1)}:=0$.

For fix $n$, denote
$$J_n:=2({\bf z}^{n})^T{\bf Q}^{(n)}({\cal D}_\tau^\beta {\bf z})^{n-\theta}-\sum_{k=1}^nA_{n-k}^{(n)}\nabla_\tau [({\bf z}^k)^T{\bf Q}^{(n)}{\bf z}^k].$$
Then
\begin{align*}
J_n=& \sum_{k=1}^nA_{n-k}^{(n)}\left[2({\bf z}^{n})^T{\bf Q}^{(n)}({\bf z}^{k}-{\bf z}^{k-1})-({\bf z}^{k}+{\bf z}^{k-1})^T{\bf Q}^{(n)}({\bf z}^{k}-{\bf z}^{k-1}) \right]\\
=&\sum_{k=1}^nA_{n-k}^{(n)}\left[\left(2{\bf z}^{n}-({\bf z}^{k}+{\bf z}^{k-1})\right)^T{\bf Q}^{(n)}({\bf z}^{k}-{\bf z}^{k-1}) \right]\\
=&\sum_{k=1}^nA_{n-k}^{(n)}({\bf z}^{k}-{\bf z}^{k-1})^T{\bf Q}^{(n)}({\bf z}^{k}-{\bf z}^{k-1})+2\sum_{k=1}^nA_{n-k}^{(n)}\sum_{j=k+1}^n({\bf z}^{j}-{\bf z}^{j-1})^T{\bf Q}^{(n)}({\bf z}^{k}-{\bf z}^{k-1}) \\
=&\sum_{k=1}^nA_{n-k}^{(n)}({\bf z}^{k}-{\bf z}^{k-1})^T{\bf Q}^{(n)}({\bf z}^{k}-{\bf z}^{k-1})+2\sum_{j=2}^n\sum_{k=1}^{j-1}A_{n-k}^{(n)}({\bf z}^{j}-{\bf z}^{j-1})^T{\bf Q}^{(n)}({\bf z}^{k}-{\bf z}^{k-1}).
\end{align*}
where the identity $2{\bf z}^{n}-({\bf z}^{k}+{\bf z}^{k-1})={\bf z}^{k}-{\bf z}^{k-1}+2\sum_{j=k+1}^n({\bf z}^{j}-{\bf z}^{j-1})$ has been employed in the third equality.

Next, introduce the quantities
$${\bf w}^j:=\sum_{k=1}^{j}A_{n-k}^{(n)}({\bf z}^{k}-{\bf z}^{k-1}) \quad \mbox{and}\quad
B_j:=\frac1{A_{n-j}^{(n)}}\quad \mbox{for}~1\leq j\leq n.$$
It holds that ${\bf z}^{j}-{\bf z}^{j-1}=B_j({\bf w}^j-{\bf w}^{j-1})$ for $2\leq j\leq n$, and $B_1\geq B_2\geq \cdots\geq B_n$ (according to the monotone property in {\bf A1}). Then
\begin{align*}
J_n=&B_1({\bf w}^1)^T{\bf Q}^{(n)}{\bf w}^1+\sum_{j=2}^nB_j({\bf w}^j-{\bf w}^{j-1})^T{\bf Q}^{(n)}({\bf w}^j-{\bf w}^{j-1})+2\sum_{j=2}^nB_j({\bf w}^j-{\bf w}^{j-1})^T{\bf Q}^{(n)}{\bf w}^{j-1}\\
=&B_1({\bf w}^1)^T{\bf Q}^{(n)}{\bf w}^1+\sum_{j=2}^nB_j\left[({\bf w}^j)^T{\bf Q}^{(n)}{\bf w}^j-({\bf w}^{j-1})^T{\bf Q}^{(n)}{\bf w}^{j-1} \right]\\
=&B_n({\bf w}^n)^T{\bf Q}^{(n)}{\bf w}^n+\sum_{j=1}^{n-1}(B_j-B_{j+1})({\bf w}^j)^T{\bf Q}^{(n)}{\bf w}^j\\
\geq& B_n({\bf w}^n)^T{\bf Q}^{(n)}{\bf w}^n,
\end{align*}
because ${\bf Q}^{(n)}$ is a positive definite matrix. Hence, the inequality \eqref{zQDz11} is valid since ${\bf w}^n=({\cal D}_\tau^\beta {\bf z})^{n-\theta}$ and $B_n=1/A_0^{(n)}$.  Similarly, it is easy to trace the remaining parts of \cite[Lemma A.1]{LiaoGronwall} to check  inequality \eqref{zQDz12}.

According to \cite[Lemma 4.1]{LiaoGronwall}  and \cite[Corollary 2.3]{LiaoSecondOrder}, with the maximum time-step ratio $\rho=7/4$, we have
$$\frac{1-\theta}{A_0^{(n)}}-\frac{\theta}{A_0^{(n)}-A_1^{(n)}}\geq 0,$$
which  further leads to  \eqref{zQDz1} by a simple combination of \eqref{zQDz11} and \eqref{zQDz12}.

\subsection{Truncation error analysis}\label{TEA}

The truncation errors in \eqref{TE-sub} and \eqref{TE-dw} are given by
\begin{align*}
&{\cal T}_u({\bf x}_h,t_{n-\theta}):={\cal D}_t^\alpha u({\bf x}_h,t_{n-\theta})-\left({\cal D}_\tau^\alpha u({\bf x}_h,\cdot)\right)^{n-\theta},\\
&{\cal T}_A({\bf x}_h,t_{n-\theta}):={\cal A}_h^{n-\theta}\left\{u({\bf x}_h,t_{n-\theta})-\left[(1-\theta)u({\bf x}_h,t_n)+\theta u({\bf x}_h,t_{n-1}) \right]\right\},\\
&{\cal T}_{\tilde u}({\bf x}_h,t_{n-\theta}):={\cal D}_t^\beta {\tilde u}({\bf x}_h,t_{n-\theta})-\left({\cal D}_\tau^\beta {\tilde u}({\bf x}_h,\cdot)\right)^{n-\theta},\\
&{\cal T}_v({\bf x}_h,t_{n-\theta}):={\cal D}_t^\beta v({\bf x}_h,t_{n-\theta})-\left({\cal D}_\tau^\beta v({\bf x}_h,\cdot)\right)^{n-\theta},\\
&{\cal S}({\bf x}_h,t_{n-\theta}):=({\cal A}u)({\bf x}_h,t_{n-\theta})-{\cal A}_h^{n-\theta}u({\bf x}_h,t_{n-\theta}),
\end{align*}
for ${\bf x}_h\in\Omega_h$ and $1\leq n\leq N$.

We first study the spatial error ${\cal S}({\bf x}_h,t_{n-\theta})$.
By the Taylor expansion (see also \cite[eq. (31)]{SunNMPDE2001}), we can take a continuous function ${\xi}^n({\bf x})$  such that
%\begin{align*}
%{\xi}^n({\bf x})=&\frac{h_x^2}{16}\int_0^1\left[(pu_x)_{xxx}\left(x+\frac{h_x}{2}s,y,t_{n-\theta}\right)-(pu_x)_{xxx}\left(x-\frac{h_x}{2}s,y,t_{n-\theta}\right)\right](1-s)\zd s\\
%&+\frac{h_x^2}{32}\int_0^1\left[ (p{\hat u})_x\left(x+\frac{h_x}{2}\eta,y,t_{n-\theta}\right)+(p{\check u})_x\left(x-\frac{h_x}{2}\eta,y,t_{n-\theta}\right)\right]
%\end{align*}
\begin{align*}
&{\xi}^n({\bf x}_h)=\left[\partial_x(p_1\partial_x u)+\partial_y(p_1\partial_y u)\right]({\bf x}_h,t_{n-\theta})-\left\{\delta_x[(p_1)_h^{n-\theta}\delta_x]+\delta_y[(p_1)_h^{n-\theta}\delta_y]\right\}u({\bf x}_h,t_{n-\theta}),
%&[\partial_y(p_1\partial_y u)]({\bf x}_h,t_{n-\theta})-\delta_y[(p_1)_h^{n-\theta}\delta_y]u({\bf x}_h,t_{n-\theta})={\eta}^n({\bf x}_h),
\end{align*}
where ${\bf x}_h\in\Omega_h$ and $1\leq n\leq N$, and $|{\xi}^n({\bf x}_h)|\leq C(h_x^2+h_y^2)$ %, $|{\eta}^n({\bf x}_h)|\leq Ch_y^2$
provided that $\|u\|_{H^4}\leq C$ and $p_k({\bf x},\cdot)\in {\cal C}^3(\Omega)$ for $k=1,2$.

Similarly, there is a continuous function  $\eta^n({\bf x})$  such that
\begin{align*}
{\eta}^n({\bf x}_h)=\left(p_3\partial_x u+p_4\partial_y u)\right]({\bf x}_h,t_{n-\theta})-\left[(p_1)_h^{n-\theta}\delta_{\hat x}+(p_4)_h^{n-\theta}\delta_y\right]u({\bf x}_h,t_{n-\theta}),
\end{align*}
where ${\bf x}_h\in\Omega_h$ and $1\leq n\leq N$, and $|{\eta}^n({\bf x}_h)|\leq C(h_x^2+h_y^2)$ provided $\|u\|_{H^3}\leq C$ and $|p_k({\bf x},\cdot)|\leq C$ for $k=3,4$.

Hence, based on {\bf V2} and the regularity assumption \eqref{regularity1}, we have
\begin{align}\label{S}
\left|{\cal S}({\bf x}_h,t_{n-\theta})\right|={\cal O}(h_x^2+h_y^2).
\end{align}
By the Taylor expansion with integral remainder, we  further get that
$$\delta_x\xi^n(x_{i+\frac12},y_j)=\frac{1}{2}\int_0^1\left[\xi_x^n\left(x_{i+\frac12}+\frac{h_x}{2}s,y_j\right)+\xi_x^n\left(x_{i+\frac12}-\frac{h_x}{2}s,y_j\right)\right](1-s)\zd s,$$
for $0\leq i\leq M_x,~1\leq j\leq M_y-1$, and
$$\delta_y\xi^n(x_i,y_{j+\frac12})=\frac{1}{2}\int_0^1\left[\xi_y^n\left(x_i,y_{j+\frac12}+\frac{h_y}{2}s\right)+\xi_y^n\left(x_i,y_{j+\frac12}-\frac{h_y}{2}s\right)\right](1-s)\zd s,$$
for $1\leq i\leq M_x-1,~0\leq j\leq M_y$. Similar formulations work for $\delta_x\eta^n(x_{i+\frac12},y_j)$ and $\delta_y\eta^n(x_i,y_{j+\frac12})$. Thus, under the assumptions in {\bf V2} and \eqref{regularity1},  it is easy to know that
\begin{align}\label{nab_S}
\|\nabla_h {\cal S}({\bf x}_h,t_{n-\theta})\|\leq C(h_x^2+h_y^2),\quad {\bf x}_h\in\Omega_h,~1\leq n\leq N. %\quad \mbox{and}\quad \|\nabla_h (p{\cal S})({\bf x}_h,t_{n-\theta})\|\leq C(h_x^2+h_y^2),
\end{align}
For the temporal truncation errors, according to \cite[Lemma 6.1]{LyuVong2020diffu-wave}, we have
\begin{equation}\label{TA}
\sum_{j=1}^nP_{n-j}^{(n)}\|({\cal T}_A)^{n-\theta}\|
\leq C\tau^{\min\{2,\gamma\sigma\}}.
\end{equation}
Referring to  \cite[eqs. (6.5), (6.6) and (6.8)]{LyuVong2020diffu-wave},  similar to the estimation of $\|\nabla_h {\cal S}({\bf x}_h,t_{n-\theta})\|$, we have
\begin{align}\label{nab_TA}
&\sum_{j=1}^nP_{n-j}^{(n)}\|\nabla_h({\cal T}_A)^{n-\theta}\|
\leq C\tau^{\min\{2,\gamma\sigma\}},
\\\label{nab_Tu}
&\sum_{j=1}^nP_{n-j}^{(n)}\|\nabla_h({\cal T}_u)^{n-\theta}\| \leq
 C\tau^{\min\{3-\beta,\gamma\sigma_1\}},
\\\label{nab_Ttildeu}
&\sum_{j=1}^nP_{n-j}^{(n)}\|\nabla_h({\cal T}_{\tilde u})^{n-\theta}\| \leq
 C\tau^{\min\{3-\beta,\gamma\sigma_2\}},
\\\label{nab_Tv1}
&\sum_{j=1}^nP_{n-j}^{(n)}\|({\cal T}_{v1})^{n-\theta}\| \leq
 C\tau^{\min\{3-\beta,\gamma\sigma_3\}},
\\\label{nab_Tv2}
&\sum_{j=1}^nP_{n-j}^{(n)}\|\nabla_h({\cal T}_{v2})^{n-\theta}\|
\leq C\tau^{\min\{2,\gamma\sigma_3\}},
\end{align}
for $1\leq n\leq N$, provided that assumptions in {\bf V2} and \eqref{regularity2}--\eqref{regularity3} are valid.

\end{document}